\documentclass{amsart}%
\usepackage{amsmath}
\usepackage{amssymb}
\usepackage{amsfonts}
\usepackage{upref,amsthm,amsxtra,exscale}
\usepackage{cite}
\usepackage[colorlinks=true,urlcolor=blue,
citecolor=red,linkcolor=blue,linktocpage,pdfpagelabels,
bookmarksnumbered,bookmarksopen]{hyperref}
\usepackage{epsfig,graphics,color}
\usepackage{graphicx}%
\newtheorem{theorem}{Theorem}[section]
\theoremstyle{plain}
\newtheorem{corollary}[theorem]{Corollary}
\newtheorem{lemma}[theorem]{Lemma}
\newtheorem{proposition}[theorem]{Proposition}

\numberwithin{equation}{section}

\begin{document}
\title[A scalar field equation with zero mass at infinity]{Existence of a positive solution to a nonlinear scalar field equation with zero mass at infinity}
\author{M\'{o}nica Clapp}
\address{Instituto de Matem\'{a}ticas, Universidad Nacional Aut\'{o}noma de M\'{e}xico,
Circuito Exterior, Ciudad Universitaria, 04510 Coyoac\'{a}n, CDMX, Mexico.}
\email{monica.clapp@im.unam.mx}
\author{Liliane A. Maia}
\address{Departamento de Matem\'{a}tica, UNB, 70910-900 Bras\'{\i}lia, Brazil.}
\email{lilimaia@unb.br}
\thanks{M. Clapp was supported by CONACYT grant 237661 (Mexico) and UNAM-DGAPA-PAPIIT grant
IN104315 (Mexico). L. Maia was supported by CNPq/PQ 308173/2014-7 (Brazil) and
PROEX/CAPES (Brazil).}
\date{\today}

\begin{abstract}
We establish the existence of a positive solution to the problem
\[
-\Delta u+V(x)u=f(u),\qquad u\in D^{1,2}(\mathbb{R}^{N}),
\]
for $N\geq3$, when the nonlinearity $f$ is subcritical at infinity and
supercritical near the origin, and the potential $V$ vanishes at infinity. Our
result includes situations in which the problem does not have a ground state.
Then, under a suitable decay assumption on the potential, we show that the
problem has a positive bound state.


\medskip\medskip\medskip

\end{abstract}
\maketitle

\section{Introduction}

\label{sec:introduction}

This paper is concerned with the existence of a positive solution to the
problem
\begin{equation}
\left\{
\begin{array}
[c]{l}%
-\Delta u+V(x)u=f(u),\\
u\in D^{1,2}(\mathbb{R}^{N}),
\end{array}
\right.  \tag{$\wp_V$}\label{prob}%
\end{equation}
for $N\geq3$, where the nonlinearity $f$ is subcritical at infinity and
supercritical near the origin, and the potential $V$ vanishes at infinity. Our
precise assumptions on $V$ and $f$ are stated below.

In their groundbreaking paper \cite{bl}, Berestycki and Lions considered the
case where $V\equiv\lambda$ is constant and $f$ has superlinear growth. They
showed that, if $f$ is subcritical, the problem (\ref{prob}) has a solution
for $\lambda>0$ and it does not have a solution for $\lambda<0$. They also
studied the limiting case $V\equiv0,$ which they called the \emph{zero mass}
case. They showed that, if $f$ is subcritical at infinity and supercritical
near the origin, the problem%
\begin{equation}
-\Delta u=f(u),\qquad u\in D^{1,2}(\mathbb{R}^{N}), \tag{$\wp_0$}\label{pinf}%
\end{equation}
has a ground state solution $\omega$, which is positive, radially symmetric
and decreasing in the radial direction.

The motivation for studying this type of equations came from some problems in
particle physics, related to the nonabelian gauge theory which underlies
strong interaction, called quantum chromodynamics or QCD. Their solutions give
rise to some special solutions of the pure Yang-Mills equations via 't Hooft's
Ansatz; see \cite{g}.

For a radial potential $V(\left\vert x\right\vert ),$ Badiale and Rolando
established the existence of a positive radial solution to the problem
(\ref{prob}) in \cite{br}. On the other hand, under suitable hypotheses, but
without assuming any symmetries on $V$, Benci, Grisanti and Micheletti showed
in \cite{bgm1}\ that the problem (\ref{prob}) has a positive least energy
solution if $V(x)\leq0$ for all $x\in\mathbb{R}^{N}$ and $V(x)<0$ on a set of
positive measure. They also showed that, if $V(x)\geq0$ for all $x\in
\mathbb{R}^{N}$ and $V(x)>0$ on a set of positive measure, this problem does
not have a ground state solution, i.e., the corresponding variational
functional does not attain its (least energy) mountain pass value. Other
related results may be found in \cite{ap,bgm2,bm,gm}.

The result that we present in this paper includes the existence of a positive
bound state for positive or sign changing potentials which decay to $0$ at
infinity with a suitable velocity. More precisely, we assume that $V$ and $f$
have the following properties:

\begin{enumerate}
\item[$\left(  V1\right)  $] $V\in L^{N/2}(\mathbb{R}^{N})\cap L^{r}%
(\mathbb{R}^{N})$ for some $r>N/2,$ and $\int_{\mathbb{R}^{N}}\left\vert
V^{-}\right\vert ^{N/2}<S^{N/2},$ where $V^{-}:=\min\{0,V\}$ and $S$ is the
best constant for the embedding $D^{1,2}(\mathbb{R}^{N})\hookrightarrow
L^{2^{\ast}}(\mathbb{R}^{N})$ with $2^{\ast}:=\frac{2N}{N-2}.$

\item[$\left(  V2\right)  $] There are constants $A_{0}>0$ and $\kappa
>\max\{2,N-2\}$ such that%
\[
V(x)\leq A_{0}(1+|x|)^{-\kappa}\quad\text{for all \ }x\in\mathbb{R}^{N}.
\]

\item[$\left(  f1\right)  $] $f\in\mathcal{C}^{1}[0,\infty),$ and there are
constants $A_{1}>0$ and $2<p<2^{\ast}<q$ such that, for $m=-1,0,1,$
\[
|f^{(m)}(s)|\leq\left\{
\begin{array}
[c]{cl}%
A_{1}\left\vert s\right\vert ^{p-(m+1)} & \text{if}\ \left\vert s\right\vert
\geq1,\\
A_{1}\left\vert s\right\vert ^{q-(m+1)} & \text{if}\ \left\vert s\right\vert
\leq1,
\end{array}
\right.
\]
where $f^{\left(  -1\right)  }:=F,$ $f^{(0)}:=f,$ $f^{(1)}:=f^{\prime},$ and
$F(s):=\int_{0}^{s}f(t)\mathrm{d}t.$

\item[$\left(  f2\right)  $] There is a constant $\theta>2$ such that
$0\leq\theta F(s)\leq f(s)s<f^{\prime}(s)s^{2}$ for all $s>0.$

\item[$\left(  f3\right)  $] The function $g(s):=\frac{sf^{\prime}(s)}{f(s)}$
is a decreasing function of $s>0$ and
\[
\lim_{s\rightarrow\infty}g(s)<2^{\ast}-1<\lim_{s\rightarrow0}g(s).
\]

\end{enumerate}

Our main result is the following one.

\begin{theorem}
\label{thm:main}Assume that $(V1)$-$(V2)$ and $(f1)$-$(f3)$ hold true. Then
the problem \emph{(\ref{prob})} has a positive solution.
\end{theorem}

It is easy to see that the model nonlinearity
\[
f(s):=\frac{s^{q-1}}{1+s^{q-p}}%
\]
satisfies the assumptions $(f1)$-$(f3).$

We point out that assumptions $(V1)$, $(f1)$ and $(f2)$ are quite natural and
have been also considered in previous works, in particular, in \cite{bgm1}.
Assumption $(f3)$ guarantees that the limit problem (\ref{pinf}) has a unique
positive solution. This fact, together with some fine estimates, which involve
assumption $(V2)$, allows us to show the existence of a positive bound state
for the problem (\ref{prob}) when the ground state is not attained.

The positive mass case, in which the potential $V$ tends to a positive
constant at infinity, has been widely investigated. A brief account may be
found in \cite{cm}, where a result, similar to Theorem \ref{thm:main},\ was
obtained for subcritical nonlinearities. On the other hand, except for the
case of the critical pure power nonlinearity, only few results are known for
the zero mass case.

There are several delicate issues in dealing with the zero mass case. Already
the variational formulation requires some care, because the energy space
$D^{1,2}(\mathbb{R}^{N})$ is only embedded in $L^{2^{\ast}}(\mathbb{R}^{N}).$
The growth assumptions $(f1)$ on the nonlinearity, however, provide the basic
interpolation and boundedness conditions that allow to establish the
differentiability of the variational functional and to study its compactness
properties. Benci and Fortunato, in \cite{bf}, expressed these conditions in
the framework of Orlicz spaces, which was also used and further developed in
\cite{ap,br,bgm1,bgm2,bm,bpr,gm}. The crucial facts, for our purposes, are
stated in Proposition \ref{prop:bpr}\ below.

Another sensitive issue is the lack of compactness. In the positive mass case,
a fundamental tool for dealing with it, is Lions' vanishing lemma, whose proof
relies deeply on the fact that the sequences involved are bounded in
$H^{1}(\mathbb{R}^{N}).$ Once again, assumption $(f1)$ allowed us to obtain a
suitable version of this result for sequences which are only bounded in
$D^{1,2}(\mathbb{R}^{N})$ (see Lemma \ref{lem:lions}). This new version of
Lions' vanishing lemma plays a crucial role in the proof of the splitting
lemma (Lemma \ref{lem:split}) which describes the lack of compactness of the
variational functional. When the ground state is not attained, due to the
uniqueness of the positive solution to limit problem (\ref{pinf}), the
splitting lemma provides an open interval of values at which the energy
functional satisfies the Palais-Smale condition.

We give a topological argument to establish the existence of a critical value
in this interval. This argument requires, on the one hand, some fine estimates
which are based on the precise asymptotic decay for the solutions of the limit
problem (\ref{pinf}), obtained recently by V\'{e}tois in \cite{v}, and on a
suitable deformation lemma for $\mathcal{C}^{1}$-manifolds that was proved by
Bonnet in \cite{b}.

This paper is organized as follows: In Section \ref{sec:limprob} we collect
the information that we need about the solutions of the limit problem. Section
\ref{sec:variational setting} is devoted to the study of the variational
problem and, specially, of the compactness properties of the variational
functional. In Section \ref{sec:existence}\ we derive the estimates that we
need, and we prove our main result.

\section{The limit problem}

\label{sec:limprob}We define $f(u):=-f(-u)$ for $u<0.$ Then, $f\in
\mathcal{C}^{1}(\mathbb{R})$ and it is an odd function. Note that, if $u$ is a
positive solution of the problem (\ref{prob}) for this new function, it is
also a solution of (\ref{prob}) for the original function $f.$ Hereafter, $f$
will denote this extension.

We consider the Hilbert space $D^{1,2}(\mathbb{R}^{N}):=\{u\in L^{2^{\ast}%
}(\mathbb{R}^{N}):\nabla u\in L^{2}(\mathbb{R}^{N},\mathbb{R}^{N})\}$ with its
standard scalar product and norm%
\[
\left\langle u,v\right\rangle :=\int_{\mathbb{R}^{N}}\nabla u\cdot\nabla
v,\text{\qquad}\left\Vert u\right\Vert :=\left(  \int_{\mathbb{R}^{N}%
}\left\vert \nabla u\right\vert ^{2}\right)  ^{1/2}.
\]
In this section we collect the information that we need on the positive
solutions to the limit problem (\ref{pinf}).

Since $f\in\mathcal{C}^{1}(\mathbb{R})$ and $f$ satisfies $(f1)$, a classical
result of Berestycki and Lions establishes the existence of a ground state
solution $\omega\in\mathcal{C}^{2}(\mathbb{R}^{N})$ to the problem
(\ref{pinf}), which is positive, radially symmetric and decreasing in the
radial direction; see Theorem 4 in \cite{bl}.

Observe that assumption $(f1)$ implies that $\left\vert f(s)\right\vert \leq
A_{1}\left\vert s\right\vert ^{2^{\ast}-1}$ and $\left\vert f^{\prime
}(s)\right\vert \leq A_{1}\left\vert s\right\vert ^{2^{\ast}-2}.$ Note also
that assumption $(f2)$ yields that $f(s)>0$ if $s>0.$ Therefore, a recent
result of V\'{e}tois implies that every positive solution $u$ to (\ref{pinf})
satisfies the decay estimates
\begin{equation}%
\begin{array}
[c]{c}%
A_{2}(1+\left\vert x\right\vert )^{-(N-2)}\leq\left\vert u(x)\right\vert \leq
A_{3}(1+\left\vert x\right\vert )^{-(N-2)},\\
\left\vert \nabla u(x)\right\vert \leq A_{3}(1+\left\vert x\right\vert
)^{-(N-1)},
\end{array}
\label{decay}%
\end{equation}
for some positive constants $A_{2}$ and $A_{3}$ and, moreover, $u$ is radially
symmetric and strictly radially decreasing about some point $x_{0}%
\in\mathbb{R}^{N}$; see Theorem 1.1 and Corollary 1.2 in \cite{v}.

Concerning uniqueness, Erbe and Tang showed that, if $f$ also satisfies
$(f3),$ then the problem (\ref{pinf}) has a unique fast decaying radial
solution, up to translations, where fast decaying means that $u$ is positive
and there is a constant $c\in(0,\infty)$ such that $\lim_{\left\vert
x\right\vert \rightarrow\infty}\left\vert x\right\vert ^{N-2}u(\left\vert
x\right\vert )=c;$ see Theorem 2 in \cite{et}\ and the remark in the paragraph
following it.

We summarize these results in the following statement.

\begin{proposition}
\label{prop:limit_problem}Under the assumptions $(f1)$-$(f3),$ the limit
problem \emph{(}\ref{pinf}\emph{)} has a unique positive solution $\omega$, up
to translations. Moreover, $\omega\in\mathcal{C}^{2}(\mathbb{R}^{N}),$ it is
radially symmetric and strictly decreasing in the radial direction, and it
satisfies the decay estimates \emph{(\ref{decay}).}
\end{proposition}

\section{The variational setting}

\label{sec:variational setting}

For $u,v\in D^{1,2}(\mathbb{R}^{N})$ we set%
\begin{equation}
\left\langle u,v\right\rangle _{V}:=\int_{\mathbb{R}^{N}}\nabla u\cdot\nabla
v+V(x)uv,\text{\qquad}\left\Vert u\right\Vert _{V}^{2}:=\int_{\mathbb{R}^{N}%
}\left(  |\nabla u|^{2}+V(x)u^{2}\right)  . \label{spV}%
\end{equation}
By assumption $(V1),$ these expresions are well defined and, using the Sobolev
inequality, we conclude that $\left\Vert \cdot\right\Vert _{V}$ is a norm in
$D^{1,2}(\mathbb{R}^{N})$ which is equivalent to the standard one.

Let $2<p<2^{\ast}<q$. The following proposition, combined with assumption
$(f1),$ provides the interpolation and boundedness properties that are needed
to obtain a good variational problem.

\begin{proposition}
\label{prop:bpr}Let $\alpha,\beta>0$ and $h\in\mathcal{C}^{0}(\mathbb{R})$.
Assume that $\frac{\alpha}{\beta}\leq\frac{p}{q},$ $\beta\leq q,$ and there
exists $M>0$ such that
\[
\left\vert h(s)\right\vert \leq M\min\{\left\vert s\right\vert ^{\alpha
},\left\vert s\right\vert ^{\beta}\}\text{\qquad for every }s\in
\mathbb{R}\text{.}%
\]
Then, for every $t\in\left[  \frac{q}{\beta},\frac{p}{\alpha}\right]  ,$ the
map $D^{1,2}(\mathbb{R}^{N})\rightarrow L^{t}(\mathbb{R}^{N})$ given by
$u\mapsto h(u)$ is well defined, continuous and bounded.
\end{proposition}

\begin{proof}
The decomposition $u=u1_{\Omega_{u}}+u1_{\mathbb{R}^{N}\smallsetminus
\Omega_{u}},$ where $\Omega_{u}:=\{x\in\mathbb{R}^{N}:\left\vert
u(x)\right\vert >1\},$ gives a continuous embedding of $L^{2^{\ast}%
}(\mathbb{R}^{N})$ and, hence, of $D^{1,2}(\mathbb{R}^{N}),$ into the Orlicz
space%
\[
L^{p}(\mathbb{R}^{N})+L^{q}(\mathbb{R}^{N}):=\{u:u=u_{1}+u_{2}\text{ with
}u_{1}\in L^{p}(\mathbb{R}^{N}),\text{ }u_{2}\in L^{q}(\mathbb{R}^{N})\},
\]
whose norm is defined by%
\[
\left\vert u\right\vert _{p,q}:=\inf\{\left\vert u_{1}\right\vert
_{p}+\left\vert u_{2}\right\vert _{q}:u=u_{1}+u_{2},\text{ }u_{1}\in
L^{p}(\mathbb{R}^{N}),\text{ }u_{2}\in L^{q}(\mathbb{R}^{N})\}.
\]
Therefore, our claim is a special case of Proposition 3.5 in \cite{bpr}.
\end{proof}

Let $F(u):=\int_{0}^{u}f(s)\,\mathrm{d}s\text{.}$ Assumption $(f1)$ implies
that $\left\vert F(s)\right\vert \leq A_{1}\left\vert s\right\vert ^{2^{\ast}%
}$ and $\left\vert f(s)\right\vert \leq A_{1}\left\vert s\right\vert
^{2^{\ast}-1}.$ Therefore, the functionals $\Phi$, $\Psi:D^{1,2}%
(\mathbb{R}^{N})\rightarrow\mathbb{R}$ given by
\[
\Phi(u):=\int_{\mathbb{R}^{N}}F(u),\qquad\Psi(u):=\int_{\mathbb{R}^{N}}f(u)u
\]
are well defined. Using Proposition \ref{prop:bpr} it is easy to show that
$\Phi$ is of class $\mathcal{C}^{2}$ and $\Psi$ is of class $\mathcal{C}^{1}$;
see Lemma 2.6 in \cite{bm} or Proposition 3.8 in \cite{bpr}. Hence, the
functional $I_{V}:D^{1,2}(\mathbb{R}^{N})\rightarrow\mathbb{R}$ given by
\[
I_{V}(u):=\frac{1}{2}\int_{\mathbb{R}^{N}}\left(  |\nabla u|^{2}%
+V(x)u^{2}\right)  -\int_{\mathbb{R}^{N}}F(u),
\]
is of class $\mathcal{C}^{2},$ with derivative
\[
I_{V}^{\prime}(u)v=\int_{\mathbb{R}^{N}}\left(  \nabla u\cdot\nabla
v+V(x)uv\right)  -\int_{\mathbb{R}^{N}}f(u)v,\qquad u,v\in D^{1,2}%
(\mathbb{R}^{N}),
\]
and the functional $J_{V}:D^{1,2}(\mathbb{R}^{N})\rightarrow\mathbb{R}$
defined by%
\[
J_{V}(u):=I_{V}^{\prime}(u)u=\int_{\mathbb{R}^{N}}\left(  |\nabla
u|^{2}+V(x)u^{2}\right)  -\int_{\mathbb{R}^{N}}f(u)u,
\]
is of class $\mathcal{C}^{1}.$

The solutions to the problem (\ref{prob}) are the critical points of the
functional $I_{V}$. The nontrivial ones lie on the set%
\[
\mathcal{N}_{V}:=\{u\in D^{1,2}(\mathbb{R}^{N}):u\neq0,\text{ \ }%
J_{V}(u)=0\}.
\]
We define%
\[
c_{V}:=\inf_{u\in\mathcal{N}_{V}}I_{V}(u),
\]
and we write $I_{0}$, $J_{0}$, $\mathcal{N}_{0}$ and $c_{0}$ for the previous
expressions with $V=0$.

The proofs of the next two lemmas use well known arguments. We include them
for the sake of completeness. Hereafter $C$ will denote a positive constant,
not necessarily the same one.

\begin{lemma}
\label{lem:nehari}

\begin{enumerate}
\item[(a)] There exists $\varrho>0$ such that $\Vert u\Vert_{V}\geq\varrho$
for every $u\in\mathcal{N}_{V}.$

\item[(b)] $\mathcal{N}_{V}$ is a closed $\mathcal{C}^{1}$-submanifold of
$D^{1,2}(\mathbb{R}^{N})$ and a natural constraint for the functional $I_{V}$.

\item[(c)] $c_{V}>0$.

\item[(d)] If $u\in\mathcal{N}_{V}$, the function $t\mapsto I_{V}(tu)$ is
strictly increasing in $[0,1)$ and strictly decreasing in $(1,\infty).$ In
particular,
\[
I_{V}(u)=\max_{t>0}I_{V}(tu).
\]

\end{enumerate}
\end{lemma}

\begin{proof}
(a): \ Assumption $(f1)$ implies that $\left\vert f(s)s\right\vert \leq
A_{1}\left\vert s\right\vert ^{2^{\ast}}.$ So, using Sobolev's inequality, we
get
\[
J_{V}(u)\geq\Vert u\Vert_{V}^{2}-C\int_{\mathbb{R}^{N}}\left\vert u\right\vert
^{2^{\ast}}\geq\Vert u\Vert_{V}^{2}-C\Vert u\Vert_{V}^{2^{\ast}}\qquad\forall
u\in D^{1,2}(\mathbb{R}^{N}).
\]
As $2^{\ast}>2,$ there exists $\varrho>0$ such that $J_{V}(u)>0$ if $0<\Vert
u\Vert_{V}\leq\varrho.$ This proves (a).

(b): \ It follows from (a) that $\mathcal{N}_{V}$ is closed in $D^{1,2}%
(\mathbb{R}^{N}).$ Moreover, assumption $(f2)$ yields%
\[
J_{V}^{\prime}(u)u=2\Vert u\Vert_{V}^{2}-\int_{\mathbb{R}^{N}}f^{\prime
}(u)u^{2}-\int_{\mathbb{R}^{N}}f(u)u=\int_{\mathbb{R}^{N}}\left[
f(u)-f^{\prime}(u)u\right]  u<0
\]
for every $u\in\mathcal{N}_{V}.$ This implies that $0$ is a regular value of
the restriction of $J_{V}$ to $D^{1,2}(\mathbb{R}^{N})\smallsetminus\{0\},$
which is of class $\mathcal{C}^{1}.$ Hence, $\mathcal{N}_{V}$ is a
$\mathcal{C}^{1}$-submanifold of $D^{1,2}(\mathbb{R}^{N})\ $and a natural
constraint for $I_{V}.$

(c): \ Let $u\in\mathcal{N}_{V}$. From hypothesis $(f2)$ and statement (a), we
obtain that
\begin{align*}
I_{V}(u)  &  =I_{V}(u)-\frac{1}{\theta}I_{V}^{\prime}(u)u=\left(  \frac{1}%
{2}-\frac{1}{\theta}\right)  \Vert u\Vert_{V}^{2}+\int_{\mathbb{R}^{N}}\left(
\frac{1}{\theta}f(u)u-F(u)\right) \\
&  \geq\left(  \frac{1}{2}-\frac{1}{\theta}\right)  \Vert u\Vert_{V}^{2}%
\geq\left(  \frac{1}{2}-\frac{1}{\theta}\right)  \varrho^{2}.
\end{align*}
Hence, $c_{V}>0.$

(d): \ Let $u\in\mathcal{N}_{V}.$\ Then,%
\begin{align*}
\frac{d}{dt}I_{V}(tu)  &  =\frac{1}{t}J_{V}(tu)=t\Vert u\Vert_{V}^{2}%
-\int_{\mathbb{R}^{N}}f(tu)u=t\int_{\mathbb{R}^{N}}\left(  f(u)-\frac
{f(tu)}{t}\right)  u\\
&  =t\left[  \int_{u>0}\left(  \frac{f(u)}{u}-\frac{f(tu)}{tu}\right)
u^{2}+\int_{u<0}\left(  \frac{f(u)}{u}-\frac{f(tu)}{tu}\right)  u^{2}\right]
.
\end{align*}
Property $(f2)$ implies that $\frac{f(s)}{s}$ is strictly increasing for $s>0$
and strictly decreasing for $s<0.$ Therefore $\frac{d}{dt}I_{V}(tu)>0$ if
$t\in(0,1)$ and $\frac{d}{dt}I_{V}(tu)<0$ if $t\in(1,\infty).$ This proves (d).
\end{proof}

\begin{lemma}
\label{lem:sign}If $u$ is a solution of \emph{(\ref{prob})}\ with $I_{V}%
(u)\in\lbrack c_{V},2c_{V}),$ then $u$ does not change sign.
\end{lemma}

\begin{proof}
If $u$ is a solution of (\ref{prob}) then $0=I_{V}^{\prime}(u)u^{\pm}%
=J_{V}(u^{\pm}),$ where $u^{+}:=\max\{u,0\}$ and $u^{-}:=\min\{u,0\}.$ Thus,
if $u^{+}\neq0$ and $u^{-}\neq0,$ then $u^{\pm}\in\mathcal{N}_{V}$ and%
\[
I_{V}(u)=I_{V}(u^{+})+I_{V}(u^{-})\geq2c_{V},
\]
contradicting our assumption.
\end{proof}

\begin{lemma}
\label{nonexistence}The limit problem \emph{(\ref{pinf})} does not have a
solution $u$ with $I_{0}(u)\in(c_{0},2c_{0})$.
\end{lemma}

\begin{proof}
If $u$ is a solution of (\ref{pinf}) such that $I_{0}(u)\in\lbrack
c_{0},2c_{0})$ then, by Lemma \ref{lem:sign}, $u$ does not change sign. So, by
Proposition \ref{prop:limit_problem}, we have that $u=\pm\omega,$ up to a
translation. Hence, $I_{0}(u)=c_{0}.$
\end{proof}

The following version of Lions' vanishing lemma plays a crucial role in the
proof of Lemma \ref{lem:aproxPS}\ and of the splitting lemma (Lemma
\ref{lem:split}). Its proof was inspired by that of Lemma 2 in \cite{abdf}. We
write $B_{R}(y):=\{x\in\mathbb{R}^{N}:\left\vert x-y\right\vert <R\}.$

\begin{lemma}
\label{lem:lions}If $(u_{k})$ is bounded in $D^{1,2}(\mathbb{R}^{N})$ and
there exists $R>0$ such that%
\[
\lim_{k\rightarrow\infty}\left(  \sup_{y\in\mathbb{R}^{N}}\int_{B_{R}%
(y)}\left\vert u_{k}\right\vert ^{2}\right)  =0,
\]
then $\lim_{k\rightarrow\infty}\int_{\mathbb{R}^{N}}f(u_{k})u_{k}=0$.
\end{lemma}

\begin{proof}
Fix $\varepsilon\in(0,1)$ and set $\eta:=\frac{2^{\ast}}{2}>1.$ For each $k,$
consider the function
\[
w_{k}:=\left\{
\begin{array}
[c]{ll}%
|u_{k}| & \text{if }\left\vert u_{k}\right\vert \geq\varepsilon,\\
\varepsilon^{-(\eta-1)}\left\vert u_{k}\right\vert ^{\eta} & \text{if
}\left\vert u_{k}\right\vert \leq\varepsilon.
\end{array}
\right.
\]
Observe that%
\[%
\begin{array}
[c]{cc}%
\left\vert \nabla w_{k}\right\vert ^{2}=\eta^{2}\varepsilon^{-2(\eta
-1)}\left\vert u_{k}\right\vert ^{2\left(  \eta-1\right)  }\left\vert \nabla
u_{k}\right\vert ^{2}\leq\eta^{2}\left\vert \nabla u_{k}\right\vert ^{2} &
\text{if }\left\vert u_{k}\right\vert \leq\varepsilon,
\end{array}
\]
and%
\[%
\begin{array}
[c]{ll}%
\left\vert w_{k}\right\vert ^{2}\leq\varepsilon^{-(2^{\ast}-2)}\left\vert
u_{k}\right\vert ^{2^{\ast}}\leq\left\vert u_{k}\right\vert ^{2} &
\text{if\ }\left\vert u_{k}\right\vert \leq\varepsilon,\\
\left\vert w_{k}\right\vert ^{2}\leq\left\vert u_{k}\right\vert ^{2-2^{\ast}%
}\left\vert u_{k}\right\vert ^{2^{\ast}}\leq\varepsilon^{-(2^{\ast}%
-2)}\left\vert u_{k}\right\vert ^{2^{\ast}} & \text{if }\left\vert
u_{k}\right\vert \geq\varepsilon.
\end{array}
\]
Using these inequalities we obtain that
\[
\left\vert w_{k}\right\vert ^{2}\leq\left\vert u_{k}\right\vert ^{2}%
,\qquad\left\vert w_{k}\right\vert ^{2}\leq\varepsilon^{-(2^{\ast}%
-2)}\left\vert u_{k}\right\vert ^{2^{\ast}},\qquad|\nabla w_{k}|^{2}\leq
\eta^{2}|\nabla u_{k}|^{2}.
\]
Therefore, as $(u_{k})$ is bounded in $D^{1,2}(\mathbb{R}^{N}),$ we have that
\[
\left\Vert w_{k}\right\Vert _{H^{1}(\mathbb{R}^{N})}^{2}:=\int_{\mathbb{R}%
^{N}}\left\vert \nabla w_{k}\right\vert ^{2}+\left\vert w_{k}\right\vert
^{2}\leq\int_{\mathbb{R}^{N}}\eta^{2}|\nabla u_{k}|^{2}+\int_{\mathbb{R}^{N}%
}\varepsilon^{-(2^{\ast}-2)}\left\vert u_{k}\right\vert ^{2^{\ast}}\leq C,
\]
i.e., $\left(  w_{k}\right)  $ is bounded in $H^{1}(\mathbb{R}^{N}).$
Moreover,%
\[
\lim_{k\rightarrow\infty}\left(  \sup_{y\in\mathbb{R}^{N}}\int_{B_{R}%
(y)}\left\vert w_{k}\right\vert ^{2}\right)  =0.
\]
It follows from Lions' vanishing lemma \cite[Lemma 1.21]{w} that
\[
w_{k}\rightarrow0\ \ \text{in}\ \ L^{s}(\mathbb{R}^{N})\qquad\text{for
each}\ 2<s<2^{\ast}.
\]
Now, using $(f1),$ we obtain
\begin{align*}
\left\vert \int_{\mathbb{R}^{N}}f(u_{k})u_{k}\right\vert  &  \leq A_{1}\left(
\int_{\left\vert u_{k}\right\vert \geq1}\left\vert u_{k}\right\vert ^{p}%
+\int_{\left\vert u_{k}\right\vert \leq1}\left\vert u_{k}\right\vert
^{q}\right) \\
&  \leq A_{1}\left(  \int_{\left\vert u_{k}\right\vert \geq\varepsilon
}\left\vert u_{k}\right\vert ^{p}+\int_{\varepsilon\leq\left\vert
u_{k}\right\vert \leq1}\left\vert u_{k}\right\vert ^{q}+\int_{\left\vert
u_{k}\right\vert \leq\varepsilon}\left\vert u_{k}\right\vert ^{q}\right) \\
&  \leq2A_{1}\int_{\left\vert u_{k}\right\vert \geq\varepsilon}\left\vert
u_{k}\right\vert ^{p}+A_{1}\int_{\left\vert u_{k}\right\vert \leq\varepsilon
}\left\vert u_{k}\right\vert ^{q}\\
&  \leq2A_{1}\int_{\left\vert u_{k}\right\vert \geq\varepsilon}\left\vert
w_{k}\right\vert ^{p}+A_{1}\int_{\left\vert u_{k}\right\vert \leq\varepsilon
}\left\vert u_{k}\right\vert ^{q-2^{\ast}}\left\vert u_{k}\right\vert
^{2^{\ast}}\\
&  \leq2A_{1}\int_{\mathbb{R}^{N}}\left\vert w_{k}\right\vert ^{p}%
+A_{1}\varepsilon^{q-2^{\ast}}\int_{\mathbb{R}^{N}}\left\vert u_{k}\right\vert
^{2^{\ast}}.
\end{align*}
As $(u_{k})$ is bounded in $D^{1,2}(\mathbb{R}^{N})$ and $w_{k}\rightarrow0$
in $L^{p}(\mathbb{R}^{N})$, we conclude that
\[
\left\vert \int_{\mathbb{R}^{N}}f(u_{k})u_{k}\right\vert \leq C\varepsilon
^{q-2^{\ast}}.
\]
Since $\varepsilon\in(0,1)$ was arbitrarily chosen, the statement is proved.
\end{proof}

We write $\nabla I_{V}(u)$ and $\nabla J_{V}(u)$ for the gradients of $I_{V}$
and $J_{V}$ at $u$ with respect to the scalar product (\ref{spV}).

\begin{lemma}
\label{lem:aproxPS}Let $(u_{k})$ be a sequence in $D^{1,2}(\mathbb{R}^{N})$
such that $I_{V}(u_{k})\rightarrow d>0$ and $J_{V}(u_{k})\rightarrow0.$ Then
there exist $a_{1}>a_{0}>0$ such that, after passing to a subsequence,%
\[
a_{0}\leq\Vert u_{k}\Vert_{V}\leq a_{1},\text{\qquad}a_{0}\leq\Vert\nabla
J_{V}(u_{k})\Vert_{V}\leq a_{1},\text{\qquad}\left\vert J_{V}^{\prime}%
(u_{k})u_{k}\right\vert \geq a_{0}.
\]

\end{lemma}

\begin{proof}
From assumption $(f2)$ we get that%
\[
d+o(1)=I_{V}(u_{k})\leq I_{V}(u_{k})+\int_{\mathbb{R}^{N}}F(u_{k})=\frac{1}%
{2}\Vert u_{k}\Vert_{V}^{2},
\]
and
\begin{align*}
d+o(1)  &  =I_{V}(u_{k})-\frac{1}{\theta}J_{V}(u_{k})\\
&  =\left(  \frac{1}{2}-\frac{1}{\theta}\right)  \Vert u_{k}\Vert_{V}^{2}%
+\int_{\mathbb{R}^{N}}\left[  \frac{1}{\theta}f(u_{k})u_{k}-F(u_{k})\right]
\geq\left(  \frac{1}{2}-\frac{1}{\theta}\right)  \Vert u_{k}\Vert_{V}^{2}.
\end{align*}
Hence, $(u_{k})$ is bounded and bounded away from $0$ in $D^{1,2}%
(\mathbb{R}^{N})$.

By assumption $(f1),$ for any $v\in D^{1,2}(\mathbb{R}^{N}),$ we have that%
\begin{align*}
\left\vert \int_{\mathbb{R}^{N}}\left[  f^{\prime}(u_{k})u_{k}+f(u_{k}%
)\right]  v\right\vert  &  \leq C\int_{\mathbb{R}^{N}}\left\vert
u_{k}\right\vert ^{2^{\ast}-1}\left\vert v\right\vert \leq C\left\vert
u_{k}\right\vert _{2^{\ast}}^{2^{\ast}-1}\left\vert v\right\vert _{2^{\ast}}\\
&  \leq C\left\Vert u_{k}\right\Vert ^{2^{\ast}-1}\left\Vert v\right\Vert \leq
C\left\Vert v\right\Vert _{V}.
\end{align*}
Therefore,%
\[
\left\vert \left\langle \nabla J_{V}(u_{k}),v\right\rangle _{V}\right\vert
=\left\vert 2\left\langle u_{k},v\right\rangle _{V}-\int_{\mathbb{R}^{N}%
}\left[  f^{\prime}(u_{k})u_{k}+f(u_{k})\right]  v\right\vert \leq C\left\Vert
v\right\Vert _{V}.
\]
This implies that $(\nabla J_{V}(u_{k}))$ is bounded in $D^{1,2}%
(\mathbb{R}^{N}).$ Hence, after passing to a subsequence, we have that
$\left\vert J_{V}^{\prime}(u_{k})u_{k}\right\vert \rightarrow a\geq0.$ Next,
we show that $a>0$.

As $J_{V}(u_{k})\rightarrow0$ we have that
\[
0<a_{0}^{2}\leq\Vert u_{k}\Vert_{V}^{2}=\int_{\mathbb{R}^{N}}f(u_{k}%
)u_{k}+o(1).
\]
So, by Lemma \ref{lem:lions}, there exist $\delta>0$ and a sequence $(y_{k})$
in $\mathbb{R}^{N}$ such that%
\begin{equation}
\int_{B_{1}(y_{k})}|u_{k}|^{2}=\sup_{y\in\mathbb{R}^{N}}\int_{B_{1}(y)}%
|u_{k}|^{2}>\delta. \label{eq:nontrivial}%
\end{equation}
Set $\tilde{u}_{k}:=u_{k}(\cdot+y_{k})$. Replacing $(\tilde{u}_{k})$ by a
subsequence,\ we have that $\tilde{u}_{k}\rightharpoonup u$ weakly in
$D^{1,2}(\mathbb{R}^{N})$ and $\tilde{u}_{k}\rightarrow u$ in $L_{loc}%
^{2}(\mathbb{R}^{N})$. The inequality (\ref{eq:nontrivial}) implies that
$u\neq0.$ Hence, there exists a subset $\Lambda$ of positive measure such that
$u(x)\neq0$ for every $x\in\Lambda$. Assumption $(f2)$ implies that
$f^{\prime}(s)s^{2}-f(s)s>0$ if $s\neq0$. So, using Fatou's lemma, we conclude
that
\begin{align*}
a  &  =\lim_{k\rightarrow\infty}\left\vert J_{V}^{\prime}(u_{k})u_{k}%
\right\vert =\lim_{k\rightarrow\infty}\left\vert 2\Vert u_{k}\Vert_{V}%
^{2}-\int_{\mathbb{R}^{N}}\left(  f^{\prime}(u_{k})u_{k}^{2}+f(u_{k}%
)u_{k}\right)  \right\vert \\
&  =\lim_{k\rightarrow\infty}\int_{\mathbb{R}^{N}}\left(  f^{\prime}%
(u_{k})u_{k}^{2}-f(u_{k})u_{k}\right)  =\lim_{k\rightarrow\infty}%
\int_{\mathbb{R}^{N}}\left(  f^{\prime}(\tilde{u}_{k})\tilde{u}_{k}%
^{2}-f(\tilde{u}_{k})\tilde{u}_{k}\right) \\
&  \geq\liminf_{k\rightarrow\infty}\int_{\Lambda}\left(  f^{\prime}(\tilde
{u}_{k})\tilde{u}_{k}^{2}-f(\tilde{u}_{k})\tilde{u}_{k}\right)  \geq
\int_{\Lambda}\left(  f^{\prime}(u)u^{2}-f(u)u\right)  >0.
\end{align*}
This proves that $a>0$ and, hence that $\left(  J_{V}^{\prime}(u_{k}%
)u_{k}\right)  $ is bounded away from $0$ in $\mathbb{R}$. It follows that
$\Vert\nabla J_{V}(u_{k})\Vert_{V}\geq\frac{\left\vert J_{V}^{\prime}%
(u_{k})u_{k}\right\vert }{\Vert u_{k}\Vert_{V}}\geq\widetilde{a}>0.$ This
finishes the proof.
\end{proof}

For $\sigma\in\mathbb{R}$, we set $\mathcal{M}_{\sigma}:=J_{V}^{-1}(\sigma)$
if $\sigma\neq0$ and $\mathcal{M}_{0}:=\mathcal{N}_{V}.$ If $\left\vert
\sigma\right\vert $ is small enough and $u\in\mathcal{M}_{\sigma},$ we write
$\nabla_{\mathcal{M}_{\sigma}}I_{V}(u)$ for the orthogonal projection of
$\nabla I_{V}(u)$ onto the tangent space to $\mathcal{M}_{\sigma}$ at $u.$

Recall that a sequence $(u_{k})$ in $D^{1,2}(\mathbb{R}^{N})$ is said to be a
$(PS)_{d}$\emph{-sequence for} $I_{V}$ if $I_{V}(u_{k})\rightarrow d$ and
$\nabla I_{V}(u_{k})\rightarrow0$.

\begin{lemma}
\label{lem:PS}Let $\sigma_{k}\in\mathbb{R}$ and $u_{k}\in\mathcal{M}%
_{\sigma_{k}}$ be such that $\sigma_{k}\rightarrow0,$ $I_{V}(u_{k})\rightarrow
d>0$ and $\nabla_{\mathcal{M}_{\sigma_{k}}}I_{V}(u_{k})\rightarrow0$. Then
$(u_{k})$ is a $(PS)_{d}$-sequence for $I_{V}$.
\end{lemma}

\begin{proof}
Let $t_{k}\in\mathbb{R}$ be such that
\begin{equation}
\nabla I_{V}(u_{k})=\nabla_{\mathcal{M}_{\sigma_{k}}}I_{V}(u_{k})+t_{k}\nabla
J_{V}(u_{k}). \label{eq:grad}%
\end{equation}
Taking the scalar product with $u_{k},$ we get that%
\[
\sigma_{k}=J_{V}(u_{k})=I_{V}^{\prime}(u_{k})u_{k}=\left\langle \nabla
_{\mathcal{M}_{\sigma_{k}}}I_{V}(u_{k}),u_{k}\right\rangle _{V}+t_{k}%
J_{V}^{\prime}(u_{k})u_{k}.
\]
By Lemma \ref{lem:aproxPS}\ we have that $(u_{k})$ is bounded and $\left(
J_{V}^{\prime}(u_{k})u_{k}\right)  $ is bounded away from $0$. So, as
$\sigma_{k}\rightarrow0$ and $\nabla_{\mathcal{M}_{\sigma_{k}}}I_{V}%
(u_{k})\rightarrow0,$ we conclude that $t_{k}\rightarrow0.$ Moreover, as
$(\nabla J_{V}(u_{k}))$ is bounded in $D^{1,2}(\mathbb{R}^{N}),$ from equation
(\ref{eq:grad}) we get that $\nabla I_{V}(u_{k})\rightarrow0,$ as claimed.
\end{proof}

\begin{lemma}
\label{lem:bl}If $u_{k}\rightharpoonup u$ weakly in $D^{1,2}(\mathbb{R}^{N})$,
the following statements hold true:

\begin{enumerate}
\item[(a)] $\Vert u_{k}\Vert_{V}^{2}=\Vert u_{k}-u\Vert^{2}+\Vert u\Vert
_{V}^{2}+o(1).\medskip$

\item[(b)] $\int_{\mathbb{R}^{N}}\left\vert f(u_{k})-f(u)\right\vert
\left\vert \varphi\right\vert =o(1)$ \ for every $\varphi\in\mathcal{C}%
_{c}^{\infty}(\mathbb{R}^{N}),$ $\medskip$

\item[(c)] $\int_{\mathbb{R}^{N}}F(u_{k})=\int_{\mathbb{R}^{N}}F(u_{k}%
-u)+\int_{\mathbb{R}^{N}}F(u)+o(1).\medskip$

\item[(d)] $f(u_{k})-f(u_{k}-u)\longrightarrow f(u)$ $\ $in $\left(
D^{1,2}(\mathbb{R}^{N})\right)  ^{\prime}.$
\end{enumerate}
\end{lemma}

\begin{proof}
(a): $\ $Set $v_{k}:=u_{k}-u.$ Assumption $(V1)$ states that $V\in
L^{N/2}(\mathbb{R}^{N})\cap L^{r}(\mathbb{R}^{N})$ with $r>N/2.$ As
$\eta:=\frac{2r}{r-1}<2^{\ast},$ we have that $v_{k}\rightarrow0$ in
$L_{loc}^{\eta}(\mathbb{R}^{N}).$ Given $\varepsilon>0,$ we fix $R>0$ such
that $\int_{\mathbb{R}^{N}\smallsetminus B_{R}(0)}\left\vert V\right\vert
^{N/2}\leq\varepsilon^{N/2}.$ Then,%
\begin{align*}
\int_{\mathbb{R}^{N}}\left\vert V\right\vert v_{k}^{2}  &  =\int_{B_{R}%
(0)}\left\vert V\right\vert v_{k}^{2}+\int_{\mathbb{R}^{N}\smallsetminus
B_{R}(0)}\left\vert V\right\vert v_{k}^{2}\\
&  \leq\left\vert V\right\vert _{r}\left\vert v_{k}\right\vert _{L^{\eta
}(B_{R}(0))}^{2}+\left\vert V\right\vert _{L^{N/2}(\mathbb{R}^{N}%
\smallsetminus B_{R}(0))}\left\Vert v_{k}\right\Vert ^{2}\leq C\varepsilon
\end{align*}
for $k$ large enough. It follows that%
\[
\Vert u_{k}\Vert_{V}^{2}=\Vert u_{k}-u\Vert_{V}^{2}+\Vert u\Vert_{V}%
^{2}+o(1)=\Vert u_{k}-u\Vert^{2}+\Vert u\Vert_{V}^{2}+o(1).
\]

(b): \ Let $s,t\in\mathbb{R}.$ By the mean value theorem, there exists
$\zeta\in(0,1)$ such that%
\begin{align}
\left\vert f(s+t)-f(s)\right\vert  &  =\left\vert f^{\prime}(s+\zeta
t)\right\vert \left\vert t\right\vert \leq A_{1}\min\{\left\vert s+\zeta
t\right\vert ^{p-2},\left\vert s+\zeta t\right\vert ^{q-2}\}\left\vert
t\right\vert \label{eq:f}\\
&  \leq A_{1}\min\{\left(  \left\vert s\right\vert +\left\vert t\right\vert
\right)  ^{p-2},\left(  \left\vert s\right\vert +\left\vert t\right\vert
\right)  ^{q-2}\}\left\vert t\right\vert \nonumber\\
&  =h(\left\vert s\right\vert +\left\vert t\right\vert )\left\vert
t\right\vert ,\nonumber
\end{align}
where $h(s):=A_{1}\min\{\left\vert s\right\vert ^{p-2},\left\vert s\right\vert
^{q-2}\}.$ Applying Proposition \ref{prop:bpr} to this function, we get that
$\{h(\left\vert u\right\vert +\left\vert u_{k}-u\right\vert )\}$ is bounded in
$L^{p/(p-2)}(\mathbb{R}^{N}).$ So, as $u_{k}\rightarrow u$ in $L_{loc}%
^{p}(\mathbb{R}^{N})$, we conclude that%
\begin{align*}
\int_{\mathbb{R}^{N}}\left\vert f(u_{k})-f(u)\right\vert \left\vert
\varphi\right\vert  &  \leq\int_{\mathbb{R}^{N}}h(\left\vert u\right\vert
+\left\vert u_{k}-u\right\vert )\left\vert u_{k}-u\right\vert \left\vert
\varphi\right\vert \\
&  \leq\left\vert h(\left\vert u\right\vert +\left\vert u_{k}-u\right\vert
)\right\vert _{\frac{p}{p-2}}\left\vert \varphi\right\vert _{p}\left(
\int_{\text{supp}(\varphi)}\left\vert u_{k}-u\right\vert ^{p}\right)
^{1/p}=o(1).
\end{align*}

(c): \ Arguing as in (b), we have that%
\[
\left\vert F(s+t)-F(s)\right\vert \leq H(\left\vert s\right\vert +\left\vert
t\right\vert )\left\vert t\right\vert
\]
for all $s,t\in\mathbb{R},$ where $H(s):=A_{1}\min\{\left\vert s\right\vert
^{p-1},\left\vert s\right\vert ^{q-1}\}.$ Let $\varepsilon>0$ and set
$v_{k}:=u_{k}-u.$ Then, noting that $\left\vert F(s)\right\vert \leq
A_{1}\left\vert s\right\vert ^{2^{\ast}}$ and using Proposition \ref{prop:bpr}%
, we may choose $R>1$ such that
\begin{align*}
&  \int_{\left\vert x\right\vert >R}\left\vert F(u_{k})-F(v_{k}%
)-F(u)\right\vert \leq\int_{\left\vert x\right\vert >R}\left\vert
F(u_{k})-F(v_{k})\right\vert +\int_{\left\vert x\right\vert >R}\left\vert
F(u)\right\vert \\
&  \qquad\leq\int_{\left\vert x\right\vert >R}H(\left\vert v_{k}\right\vert
+\left\vert u\right\vert )\left\vert u\right\vert +A_{1}\int_{\left\vert
x\right\vert >R}\left\vert u\right\vert ^{2^{\ast}}\\
&  \qquad\leq\left\vert H(\left\vert v_{k}\right\vert +\left\vert u\right\vert
)\right\vert _{2^{\ast}/(2^{\ast}-1)}\left(  \int_{\left\vert x\right\vert
>R}\left\vert u\right\vert ^{2^{\ast}}\right)  ^{1/2^{\ast}}+A_{1}%
\int_{\left\vert x\right\vert >R}\left\vert u\right\vert ^{2^{\ast}}\\
&  \qquad\leq C\left(  \int_{\left\vert x\right\vert >R}\left\vert
u\right\vert ^{2^{\ast}}\right)  ^{1/2^{\ast}}+A_{1}\int_{\left\vert
x\right\vert >R}\left\vert u\right\vert ^{2^{\ast}}<\varepsilon.
\end{align*}
On the other hand, as $u_{k}\rightarrow u$ in $L_{loc}^{p}(\mathbb{R}^{N})$,
we have that%
\begin{align*}
&  \int_{\left\vert x\right\vert \leq R}\left\vert F(u_{k})-F(v_{k}%
)-F(u)\right\vert \\
&  \qquad\leq\int_{\left\vert x\right\vert \leq R}\left\vert F(u_{k}%
)-F(u)\right\vert +\int_{1\leq\left\vert x\right\vert \leq R}\left\vert
F(v_{k})\right\vert +\int_{\left\vert x\right\vert \leq1}\left\vert
F(v_{k})\right\vert \\
&  \qquad\leq\int_{\left\vert x\right\vert \leq R}H(\left\vert u_{k}%
\right\vert +\left\vert u\right\vert )\left\vert v_{k}\right\vert +A_{1}%
\int_{1\leq\left\vert x\right\vert \leq R}\left\vert v_{k}\right\vert
^{p}+A_{1}\int_{\left\vert x\right\vert \leq1}\left\vert v_{k}\right\vert
^{q}\\
&  \qquad\leq\left\vert H(\left\vert u_{k}\right\vert +\left\vert u\right\vert
)\right\vert _{p/(p-1)}\left(  \int_{\left\vert x\right\vert \leq R}\left\vert
v_{k}\right\vert ^{p}\right)  ^{1/p}+C\int_{\left\vert x\right\vert \leq
R}\left\vert v_{k}\right\vert ^{p}\\
&  \qquad\leq C\left(  \int_{\left\vert x\right\vert \leq R}\left\vert
v_{k}\right\vert ^{p}\right)  ^{1/p}+C\int_{\left\vert x\right\vert \leq
R}\left\vert v_{k}\right\vert ^{p}<\varepsilon
\end{align*}
if $k$ is large enough. This proves (c).

(d): \ Let $\varphi\in\mathcal{C}_{c}^{\infty}(\mathbb{R}^{N})$ and $R>0.$ Set
$h(s):=A_{1}\min\{\left\vert s\right\vert ^{p-2},\left\vert s\right\vert
^{q-2}\}.$ From (\ref{eq:f}) and Proposition \ref{prop:bpr} we get that%
\begin{align*}
&  \int_{\left\vert x\right\vert >R}\left\vert f(u_{k})-f(u_{k}-u)\right\vert
\left\vert \varphi\right\vert \leq\int_{\left\vert x\right\vert >R}%
h(\left\vert u_{k}\right\vert +\left\vert u\right\vert )\left\vert
u\right\vert \left\vert \varphi\right\vert \\
&  \leq\left\vert h(\left\vert u_{k}\right\vert +\left\vert u\right\vert
)\right\vert _{2^{\ast}/(2^{\ast}-2)}\left(  \int_{\left\vert x\right\vert
>R}\left\vert u\right\vert ^{2^{\ast}}\right)  ^{1/2^{\ast}}\left\vert
\varphi\right\vert _{2^{\ast}}\leq C\left(  \int_{\left\vert x\right\vert
>R}\left\vert u\right\vert ^{2^{\ast}}\right)  ^{1/2^{\ast}}\left\Vert
\varphi\right\Vert .
\end{align*}
Moreover, as $\left\vert f(u)\right\vert \leq A_{1}\left\vert u\right\vert
^{2^{\ast}-1},$ we have that%
\[
\int_{\left\vert x\right\vert >R}\left\vert f(u)\right\vert \left\vert
\varphi\right\vert \leq C\left(  \int_{\left\vert x\right\vert >R}\left\vert
u\right\vert ^{2^{\ast}}\right)  ^{(2^{\ast}-1)/2^{\ast}}\left\Vert
\varphi\right\Vert .
\]
Thus, given $\varepsilon>0$, we may choose $R>0$ large enough so that%
\[
\int_{\left\vert x\right\vert >R}\left\vert f(u_{k})-f(u_{k}%
-u)-f(u)\right\vert \left\vert \varphi\right\vert \leq\varepsilon\left\Vert
\varphi\right\Vert .
\]
Next, we fix $\delta\in(0,1)$ such that $\eta:=2^{\ast}\delta\in(p,2^{\ast})$
and $\nu:=\frac{\eta}{\eta-1-\delta}\in\left[  \frac{q}{q-2},\frac{p}%
{p-2}\right]  .$ As $u_{k}\rightarrow u$ strongly in $L_{loc}^{\eta
}(\mathbb{R}^{N}),$ from (\ref{eq:f}) and Proposition \ref{prop:bpr} we get
that%
\[
\int_{\left\vert x\right\vert \leq R}\left\vert f(u_{k})-f(u)\right\vert
\left\vert \varphi\right\vert \leq\left\vert h(\left\vert u\right\vert
+\left\vert u_{k}-u\right\vert )\right\vert _{\nu}\left(  \int_{\left\vert
x\right\vert \leq R}\left\vert u_{k}-u\right\vert ^{\eta}\right)  ^{1/\eta
}\left\vert \varphi\right\vert _{2^{\ast}}\leq\varepsilon\left\Vert
\varphi\right\Vert
\]
for $k$ large enough and, similarly,%
\[
\int_{\left\vert x\right\vert \leq R}\left\vert f(u_{k}-u)\right\vert
\left\vert \varphi\right\vert \leq\left\vert h(\left\vert u_{k}-u\right\vert
)\right\vert _{\nu}\left(  \int_{\left\vert x\right\vert \leq R}\left\vert
u_{k}-u\right\vert ^{\eta}\right)  ^{1/\eta}\left\vert \varphi\right\vert
_{2^{\ast}}\leq\varepsilon\left\Vert \varphi\right\Vert .
\]
Therefore,%
\[
\left\vert \int_{\mathbb{R}^{N}}\left(  f(u_{k})-f(u_{k}-u)-f(u)\right)
\varphi\right\vert \leq\varepsilon\left\Vert \varphi\right\Vert \text{\qquad
for }k\text{ large enough.}%
\]
This proves the claim.
\end{proof}

The following lemma is stated in \cite{bgm1} but its proof contains a gap.
This can be fixed with the help of Lemma \ref{lem:lions}. We give the details.

\begin{lemma}
[Splitting lemma]\label{lem:split}Let $(u_{k})$ be a bounded $(PS)_{d}%
$-sequence for $I_{V}.$ Then, after replacing $(u_{k})$ by a subsequence,
there exists a solution $u$ of problem \emph{(\ref{prob})}, a number
$m\in\mathbb{N}\cup\left\{  0\right\}  $, $m$ nontrivial solutions
$w_{1},...,w_{m}$ to the limit problem \emph{(\ref{pinf}) }and $m$ sequences
of points $(y_{j,k})\in\mathbb{R}^{N},$ $1\leq j\leq m$, satisfying

\begin{itemize}
\item[(i)] $|y_{j,k}|\rightarrow\infty,$ \ and $|y_{j,k}-y_{i,k}%
|\rightarrow\infty$ if $i\neq j$,\medskip

\item[(ii)] $u_{k}-\sum_{i=1}^{m}w_{i}(\cdot-y_{i,k})\rightarrow u$ \ in
$\ D^{1,2}(\mathbb{R}^{N})$,\medskip

\item[(iii)] $d=I_{V}(u)+\sum_{i=1}^{m}I_{0}(w_{i})$.
\end{itemize}
\end{lemma}

\begin{proof}
Passing to a subsequence, we have that $u_{k}\rightharpoonup u$ weakly in
$D^{1,2}(\mathbb{R}^{N}).$ It follows from Lemma \ref{lem:bl} that%
\begin{align*}
o(1)  &  =I_{V}^{\prime}(u_{k})\varphi=\left\langle u_{k},\varphi\right\rangle
_{V}-\int_{\mathbb{R}^{N}}f(u_{k})\varphi\\
&  =\left\langle u,\varphi\right\rangle _{V}-\int_{\mathbb{R}^{N}}%
f(u)\varphi+o(1)=I_{V}^{\prime}(u)\varphi+o(1)
\end{align*}
for every $\varphi\in\mathcal{C}_{c}^{\infty}(\mathbb{R}^{N}).$ Hence, $u$
solves (\ref{prob})\emph{.} Set $u_{1,k}:=u_{k}-u$. Then, $u_{1,k}%
\rightharpoonup0$ weakly in $D^{1,2}(\mathbb{R}^{N}).$ Moreover, Lemma
\ref{lem:bl} implies that%
\begin{equation}
\left.
\begin{array}
[c]{l}%
I_{V}(u_{k})=I_{0}(u_{1,k})+I_{V}(u)+o(1),\\
o(1)=I_{V}^{\prime}(u_{k})=I_{0}^{\prime}(u_{1,k})+o(1)\text{ \ in }\left(
{D^{1,2}(\mathbb{R}^{N})}\right)  ^{\prime}.
\end{array}
\right.  \label{eq:split1}%
\end{equation}
If $u_{1,k}\rightarrow0$ strongly in $D^{1,2}(\mathbb{R}^{N}),$ the proof is
finished. So assume it does not. Then, as $I_{V}^{\prime}(u_{1,k}%
)u_{1,k}\rightarrow0,$ after passing to a subsequence, we have that
\[
0<C\leq\Vert u_{1,k}\Vert_{V}^{2}=\int_{\mathbb{R}^{N}}f(u_{1,k}%
)u_{1,k}+o(1).
\]
So, by Lemma \ref{lem:lions}, there exist $\delta>0$ and a sequence
$(y_{1,k})$ in $\mathbb{R}^{N}$ such that%
\begin{equation}
\int_{B_{1}(y_{1,k})}|u_{1,k}|^{2}=\sup_{y\in\mathbb{R}^{N}}\int_{B_{1}%
(y)}|u_{1,k}|^{2}>\delta. \label{eq:nontrivial2}%
\end{equation}
Set $v_{k}:=u_{1,k}(\,\cdot\,+y_{1,k})$. Passing to a subsequence,\ we have
that $v_{k}\rightharpoonup w_{1}$ weakly in $D^{1,2}(\mathbb{R}^{N})$ and
$v_{k}\rightarrow w_{1}$ in $L_{loc}^{2}(\mathbb{R}^{N})$. The inequality
(\ref{eq:nontrivial2}) implies that $w_{1}\neq0$ and, as $u_{1,k}%
\rightharpoonup0$ weakly in $D^{1,2}(\mathbb{R}^{N}),$ we conclude that
$\left\vert y_{1,k}\right\vert \rightarrow\infty.$ Next, we show that $w_{1}$
is a solution to the limit problem (\ref{pinf}). Let $\varphi\in
\mathcal{C}_{c}^{\infty}(\mathbb{R}^{N})$ and set $\varphi_{k}:=\varphi
(\,\cdot\,-y_{1,k}).$ Using Lemma \ref{lem:bl} and performing a change of
variable, we obtain that
\[
I_{0}^{\prime}(w_{1})\varphi+o(1)=I_{0}^{\prime}(v_{k})\varphi=I_{0}^{\prime
}(u_{1,k})\varphi_{k}=o(1).
\]
This proves that $w_{1}$ solves the problem (\ref{pinf}). Moreover, Lemma
\ref{lem:bl} implies that%
\begin{align*}
I_{0}(v_{k})  &  =I_{0}(v_{k}-w_{1})+I_{0}(w_{1})+o(1),\\
o(1)  &  =I_{V}^{\prime}(v_{k})=I_{0}^{\prime}(v_{k}-w_{1})+o(1)\text{ \ in
}\left(  {D^{1,2}(\mathbb{R}^{N})}\right)  ^{\prime}.
\end{align*}
Set $u_{2,k}:=u_{1,k}-w_{1}(\,\cdot\,-y_{1,k})=u_{k}-u-w_{1}(\,\cdot
\,-y_{1,k}).$ Then, $u_{2,k}\rightharpoonup0$ weakly in $D^{1,2}%
(\mathbb{R}^{N})$ and, after a change of variable, from the identities
(\ref{eq:split1}) and (\ref{eq:split2}) we obtain that%
\begin{equation}
\left.
\begin{array}
[c]{l}%
I_{0}(u_{2,k})=I_{0}(u_{1,k})-I_{0}(w_{1})+o(1)=I_{V}(u_{k})-I_{V}%
(u)-I_{0}(w_{1})+o(1),\\
I_{0}^{\prime}(u_{2,k})=I_{0}^{\prime}(u_{1,k})+o(1)=I_{V}^{\prime}%
(u_{k})+o(1)=o(1)\text{ \ in }\left(  {D^{1,2}(\mathbb{R}^{N})}\right)
^{\prime}.
\end{array}
\right.  \label{eq:split2}%
\end{equation}
If $u_{2,k}\rightarrow0$ strongly in $D^{1,2}(\mathbb{R}^{N}),$ the proof is
finished. If not, we repeat the argument. After a finite number of steps, we
will arrive to a sequence $(u_{m+1,k})$ which converges strongly to $0$ in
$D^{1,2}(\mathbb{R}^{N}).$ This finishes the proof.
\end{proof}

\begin{corollary}
[Compactness]\label{cor:compactness}If $c_{V}$ is not attained by $I_{V}$ on
$\mathcal{N}_{V}$, then the following statements hold true.

\begin{enumerate}
\item[(a)] $c_{V}\geq c_{0}.$

\item[(b)] If $\sigma_{k}\in\mathbb{R}$ and $u_{k}\in\mathcal{M}_{\sigma_{k}}$
are such that $\sigma_{k}\rightarrow0,$ $I_{V}(u_{k})\rightarrow d\in
(c_{0},2c_{0})$ and $\nabla_{\mathcal{M}_{\sigma_{k}}}I_{V}(u_{k}%
)\rightarrow0$, then $(u_{k})$ contains a convergent subsequence.
\end{enumerate}
\end{corollary}

\begin{proof}
(a): \ Let $(u_{k})$ be a minimizing sequence for $I_{V}$ on $\mathcal{N}%
_{V}.$ By Ekeland's variational principle and Lemma \ref{lem:aproxPS}, we may
assume that $(u_{k})$ is a bounded $(PS)_{c_{V}}$-sequence for $I_{V}$. As
$c_{V}$ is not attained, the splitting lemma implies that $c_{V}\geq c_{0}.$

(b): \ By Lemmas \ref{lem:aproxPS} and Lemma \ref{lem:PS} we have that
$(u_{k})$ is a bounded $(PS)_{d}$-sequence for $I_{V}$. Arguing by
contradiction, assume that $(u_{k})$ does not contain a convergent
subsequence. Then, the splitting lemma yields a solution $w$ of the limit
problem (\ref{pinf})\ with $d=I_{0}(w)$, contradicting Lemma
\ref{nonexistence}. This proves our claim.
\end{proof}

\section{Existence of a positive solution}

\label{sec:existence}The proof of our main result requires some delicate
estimates. The following lemma will help us obtain them.

\begin{lemma}
\label{power}

\begin{enumerate}
\item[(a)] If $y_{0},y\in\mathbb{R}^{N},$ $y_{0}\neq y,$ and $\alpha$ and
$\beta$ are positive constants such that $\alpha+\beta>N$, then there exists
$C_{1}=C_{1}(\alpha,\beta,\lvert y-y_{0}\rvert)>0$ such that
\[
\int_{\mathbb{R}^{N}}\frac{\mathrm{d}x}{(1+|x-Ry_{0}|)^{\alpha}%
(1+|x-Ry|)^{\beta}}\leq C_{1}R^{-\mu}%
\]
for all $R\geq1,$ where $\mu:=\min\{\alpha,\beta,\alpha+\beta-N\}$.

\item[(b)] If $y_{0},y\in\mathbb{R}^{N}\smallsetminus\{0\},$ and $\kappa$ and
$\gamma$ are positive constants such that $\kappa+2\gamma>N$, then there
exists $C_{2}=C_{2}(\kappa,\gamma,\lvert y_{0}\rvert,\lvert y\rvert)>0$ such
that
\[
\int_{\mathbb{R}^{N}}\frac{\mathrm{d}x}{(1+|x|)^{\kappa}(1+|x-Ry_{0}%
|)^{\gamma}(1+|x-Ry|)^{\gamma}}\leq C_{2}R^{-\tau},
\]
for all $R\geq1,$ where $\tau:=\min\{\kappa,2\gamma,\kappa+2\gamma-N\}$.
\end{enumerate}
\end{lemma}

\begin{proof}
(a): \ After a suitable translation, we may assume that $y=-y_{0}.$ Let
$2\rho:=\lvert y-y_{0}\rvert>0.$ In the following, $C$ will denote different
positive constants which depend on $\alpha$, $\beta$ and $\rho.$ If
$\left\vert x-Ry_{0}\right\vert \leq\rho R,$ then $\left\vert x-Ry\right\vert
\geq\rho R.$ Hence
\begin{align*}
&  \int_{B_{\rho R}(Ry_{0})}\frac{\mathrm{d}x}{(1+|x-Ry_{0}|)^{\alpha
}(1+|x-Ry|)^{\beta}}=\int_{B_{\rho R}(Ry_{0})}\frac{\mathrm{d}x}%
{(1+|x-Ry_{0}|)^{\alpha}(\rho R)^{\beta}}\\
&  =CR^{-\beta}\int_{B_{\rho R}(0)}\frac{\mathrm{d}x}{(1+|x|)^{\alpha}}\leq
C(R^{-\beta}+R^{N-(\alpha+\beta)})\leq CR^{-\mu}.
\end{align*}
Similarly,
\[
\int_{B_{\rho R}(Ry)}\frac{\mathrm{d}x}{(1+|x-Ry_{0}|)^{\alpha}%
(1+|x-Ry|)^{\beta}}\leq C(R^{-\alpha}+R^{N-(\alpha+\beta)})\leq CR^{-\mu}.
\]
Let%
\begin{align*}
H^{+}  &  :=\{z\in\mathbb{R}^{N}:\left\vert z-Ry\right\vert \geq\left\vert
z-Ry_{0}\right\vert \},\\
H^{-}  &  :=\{z\in\mathbb{R}^{N}:\left\vert z-Ry\right\vert \leq\left\vert
z-Ry_{0}\right\vert \}.
\end{align*}
Setting $x=Rz$ we obtain%
\begin{align*}
&  \int_{H^{+}\smallsetminus B_{\rho R}(Ry_{0})}\frac{\mathrm{d}%
x}{(1+|x-Ry_{0}|)^{\alpha}(1+|x-Ry|)^{\beta}}\\
&  \leq\int_{H^{+}\smallsetminus B_{\rho R}(Ry_{0})}\frac{\mathrm{d}%
x}{(1+|x-Ry_{0}|)^{\alpha+\beta}}\\
&  \leq\int_{H^{+}\smallsetminus B_{\rho}(y_{0})}\frac{R^{N}\mathrm{d}%
x}{(R|z-y_{0}|)^{\alpha+\beta}}=CR^{N-(\alpha+\beta)}\leq CR^{-\mu}.
\end{align*}
Similarly,%
\[
\int_{H^{-}\smallsetminus B_{\rho R}(Ry)}\frac{\mathrm{d}x}{(1+|x-Ry_{0}%
|)^{\alpha}(1+|x-Ry|)^{\beta}}\leq CR^{-\mu}.
\]
Since $\mathbb{R}^{N}\smallsetminus\left[  B_{\rho R}(Ry_{0})\cup B_{\rho
R}(Ry)\right]  =\left[  H^{+}\smallsetminus B_{\rho R}(Ry_{0})\right]
\cup\left[  H^{-}\smallsetminus B_{\rho R}(Ry)\right]  ,$ the previous
estimates yield (a).

(b): \ From H\"{o}lder's inequality and inequality (a)\ we obtain%
\begin{align*}
&  \int_{\mathbb{R}^{N}}\frac{\mathrm{d}x}{(1+|x|)^{\kappa}(1+|x-Ry_{0}%
|)^{\gamma}(1+|x-Ry|)^{\gamma}}\\
&  \leq\left(  \int_{\mathbb{R}^{N}}\frac{\mathrm{d}x}{(1+|x|)^{\kappa
}(1+|x-Ry_{0}|)^{2\gamma}}\right)  ^{1/2}\left(  \int_{\mathbb{R}^{N}}%
\frac{\mathrm{d}x}{(1+|x|)^{\kappa}(1+|x-Ry|)^{2\gamma}}\right)  ^{1/2}\\
&  \leq C_{2}R^{-\tau},
\end{align*}
as claimed.
\end{proof}

Let $\omega$ be the positive radial ground state of the limit problem
(\ref{pinf}). Fix $y_{0}\in\mathbb{R}^{N}$ with $\left\vert y_{0}\right\vert
=1$, and let $B_{2}(y_{0}):=\{x\in\mathbb{R}^{N}:\left\vert x-y_{0}\right\vert
\leq2\}$. For $R\geq1$ and each $y\in\partial B_{2}(y_{0})$, we define%
\[
\omega_{0}^{R}:=\omega(\cdot-Ry_{0}),\qquad\omega_{y}^{R}:=\omega(\cdot-Ry),
\]
and we set
\[
\varepsilon_{R}:=\int_{\mathbb{R}^{N}}f(\omega_{0}^{R})\,\omega_{y}^{R}%
=\int_{\mathbb{R}^{N}}f(\omega(x-Ry_{0}))\,\omega(x-Ry)\,\mathrm{d}x.
\]
As before, $C$ will denote a positive constant, not necessarily the same one.

\begin{lemma}
\label{upperbound} There exists a constant $C_{3}>0$ such that
\[
\varepsilon_{R}\leq C_{3}R^{-(N-2)}%
\]
for all $y\in\partial B_{2}(y_{0})$ and all $R\geq1$.
\end{lemma}

\begin{proof}
By assumption $(f1),$ we have that $\left\vert f(s)\right\vert \leq
A_{1}\left\vert s\right\vert ^{2^{\ast}-1}$. On the other hand, from the
estimates (\ref{decay}) and Lemma \ref{power}(a) we obtain
\begin{align}
\int_{\mathbb{R}^{N}}(\omega_{0}^{R})^{2^{\ast}-1}\,\omega_{y}^{R}  &  \leq
C\int_{\mathbb{R}^{N}}\frac{\mathrm{d}x}{(1+|x-Ry_{0}|)^{N+2}(1+|x-Ry|)^{N-2}%
}\label{eq:interaction}\\
&  \leq C\,R^{-(N-2)}.\nonumber
\end{align}
Therefore,
\[
\int_{\mathbb{R}^{N}}f(\omega_{0}^{R})\,\omega_{y}^{R}\leq A_{1}%
\int_{\mathbb{R}^{N}}(\omega_{0}^{R})^{2^{\ast}-1}\,\omega_{y}^{R}\leq
C\,R^{-(N-2)}%
\]
for all $y\in\partial B_{2}(y_{0})$ and all $R\geq1$, as claimed.
\end{proof}

\begin{lemma}
\label{lowerbound}There exists a constant $C_{4}>0$ such that
\[
\int_{\mathbb{R}^{N}}f(s\omega_{0}^{R})\,t\omega_{y}^{R}\geq C_{4}R^{-(N-2)}%
\]
for all $s,t\geq\frac{1}{2}$, $y\in\partial B_{2}(y_{0})$ and $R\geq1$.
\end{lemma}

\begin{proof}
Note that, if $\left\vert x\right\vert <1,$ then, for every $y\in\partial
B_{2}(y_{0})$ and $R\geq1$,
\[
1+\left\vert x-R(y-y_{0})\right\vert <1+\left\vert x\right\vert +R\left\vert
y-y_{0}\right\vert <4R.
\]
Assumption $(f_{3})$ implies that $\frac{f(s)}{s}$ is strictly increasing for
$s>0$. Hence, so is $f.$ Performing a change of variable and using the
estimate (\ref{decay}) we obtain
\begin{align*}
\int_{\mathbb{R}^{N}}f(s\omega_{0}^{R})t\omega_{y}^{R}  &  \geq t\int
_{\mathbb{R}^{N}}f\left(  \frac{1}{2}\omega_{0}^{R}\right)  \omega_{y}^{R}%
\geq\frac{1}{2}\int_{B_{1}(Ry_{0})}f\left(  \frac{1}{2}\omega_{0}^{R}\right)
\omega_{y}^{R}\\
&  \geq\frac{1}{4}\left[  \min_{x\in B_{1}(0)}f\left(  \frac{1}{2}%
\omega(x)\right)  \right]  \int_{B_{1}(0)}\omega(x-R(y-y_{0}))\mathrm{d}x\\
&  \geq C\int_{B_{1}(0)}(1+|x-R(y-y_{0})|)^{-(N-2)}\mathrm{d}x\geq CR^{-(N-2)}%
\end{align*}
for all $s,t\geq\frac{1}{2}$, $y\in\partial B_{2}(y_{0})$ and $R\geq1$, as claimed.
\end{proof}

Note that Lemmas \ref{upperbound}\ and \ref{lowerbound}\ yield%
\begin{equation}
C_{4}R^{-(N-2)}\leq\varepsilon_{R}:=\int_{\mathbb{R}^{N}}f(\omega_{0}%
^{R})\omega_{y}^{R}\leq C_{3}R^{-(N-2)} \label{eq:epsilon}%
\end{equation}
for all $y\in\partial B_{2}(y_{0})$ and $R\geq1.$

\begin{lemma}
\label{tvm}For each $b>1$ there is a constant $C_{b}>0,$ such that%
\[
\left\vert \int_{\mathbb{R}^{N}}\left(  sf(\omega_{0}^{R})-f(s\omega_{0}%
^{R})\right)  \,\omega_{y}^{R}\right\vert \leq C_{b}\left\vert s-1\right\vert
\varepsilon_{R}%
\]
for all $s\in\lbrack0,b],$ $y\in\partial B_{2}(y_{0})$ and $R\geq1.$
\end{lemma}

\begin{proof}
Fix $t\in\mathbb{R}$ and set $g(s):=sf(t)-f(st).$ By the mean value theorem,
there exists $\zeta$ between $1$ and $s$ such that
\begin{align*}
\left\vert sf(t)-f(st)\right\vert  &  =\left\vert g(s)-g(1)\right\vert
=\left\vert f(t)-f^{\prime}(\zeta t)t\right\vert \left\vert s-1\right\vert \\
&  \leq\left(  \left\vert f(t)\right\vert +\left\vert f^{\prime}(\zeta
t)t\right\vert \right)  \left\vert s-1\right\vert \\
&  \leq A_{1}\left(  \left\vert t\right\vert ^{2^{\ast}-1}+\left\vert
\zeta\right\vert ^{2^{\ast}-2}\left\vert t\right\vert ^{2^{\ast}-1}\right)
\left\vert s-1\right\vert \\
&  \leq A_{1}(1+b^{2^{\ast}-2})\left\vert t\right\vert ^{2^{\ast}-1}\left\vert
s-1\right\vert \text{\qquad}\forall s\in\lbrack0,b],
\end{align*}
where the second-to-last inequality follows from assumption $(f1)$. So, from
the inequalities (\ref{eq:interaction}) and (\ref{eq:epsilon}) we obtain that%
\begin{align*}
\int_{\mathbb{R}^{N}}\left(  sf(\omega_{0}^{R})-f(s\omega_{0}^{R})\right)
\,\omega_{y}^{R}  &  \leq C\left\vert s-1\right\vert \int_{\mathbb{R}^{N}%
}(\omega_{0}^{R})^{2^{\ast}-1}\,\omega_{y}^{R}\\
&  \leq C\left\vert s-1\right\vert R^{-(N-2)}\leq C\left\vert s-1\right\vert
\varepsilon_{R}%
\end{align*}
for all $s\in\lbrack0,b],$ $y\in\partial B_{2}(0)$ and $R\geq1,$ as claimed.
\end{proof}

\begin{lemma}
\label{lem:V}There exists $\tau>N-2$ such that%
\[
\int_{\mathbb{R}^{N}}V^{+}\left(  \omega_{0}^{R}+\omega_{y}^{R}\right)
^{2}\leq CR^{-\tau}%
\]
for every $y\in\partial B_{2}(y_{0})$ and $R\geq1.$
\end{lemma}

\begin{proof}
From assumption $(V2),$ the estimates (\ref{decay}) and Lemma \ref{power}(b),
we immediately obtain that
\begin{align*}
\int_{\mathbb{R}^{N}}V^{+}\left(  \omega_{0}^{R}+\omega_{y}^{R}\right)  ^{2}
&  =\int_{\mathbb{R}^{N}}V^{+}\left(  \omega_{0}^{R}\right)  ^{2}%
+2\int_{\mathbb{R}^{N}}V^{+}\omega_{0}^{R}\omega_{y}^{R}+\int_{\mathbb{R}^{N}%
}V^{+}\left(  \omega_{y}^{R}\right)  ^{2}\\
&  \leq CR^{-\tau},
\end{align*}
with $\tau:=\min\{\kappa,2(N-2),\kappa+N-4\}>N-2.$
\end{proof}

For each $R\geq1$, $y\in\partial B_{2}(y_{0})$ and $\lambda\in\lbrack0,1]$, we
define
\[
z_{\lambda,y}^{R}:=\lambda\omega_{0}^{R}+(1-\lambda)\omega_{y}^{R}.
\]

\begin{lemma}
\label{projection}For each $R\geq1$, $y\in\partial B_{2}(y_{0})$ and
$\lambda\in\lbrack0,1]$, there exists a unique $T_{\lambda,y}^{R}>0$ such
that
\[
T_{\lambda,y}^{R}z_{\lambda,y}^{R}\in\mathcal{N}_{V}.
\]
Moreover, there exist $R_{0}\geq1$ and $T_{0}>2$ such that $T_{\lambda,y}%
^{R}\in\lbrack0,T_{0})$ for all $R\geq R_{0},$ $y\in\partial B_{2}(y_{0})$ and
$\lambda\in\lbrack0,1],$ and $T_{\lambda,y}^{R}$ is a continuous function of
the variables $\lambda,y$ and $R$.
\end{lemma}

\begin{proof}
The proof is the same as that of Lemma 3.2\ in \cite{cm}, with the obvious changes.
\end{proof}

\begin{lemma}
\label{lem:lambda1/2}For $\lambda=\frac{1}{2}$ we have that $T_{\lambda,y}%
^{R}\rightarrow2$ as $R\rightarrow\infty$ uniformly in $y\in\partial
B_{2}(y_{0}).$
\end{lemma}

\begin{proof}
Since $\omega$ solves (\ref{pinf}), we have that
\begin{align*}
J_{0}(\omega_{0}^{R}+\omega_{y}^{R})  &  =J_{0}(\omega_{0}^{R})+J_{0}%
(\omega_{y}^{R})+2\int_{\mathbb{R}^{N}}\nabla\omega_{0}^{R}\cdot\nabla
\omega_{y}^{R}\\
&  \quad-\int_{\mathbb{R}^{N}}[f(\omega_{0}^{R}+\omega_{y}^{R})-f(\omega
_{0}^{R})]\,\omega_{0}^{R}-\int_{\mathbb{R}^{N}}[f(\omega_{0}^{R}+\omega
_{y}^{R})-f(\omega_{y}^{R})]\,\omega_{y}^{R}\\
&  =2\int_{\mathbb{R}^{N}}\nabla\omega\cdot\nabla\omega_{y-y_{0}}^{R}%
-\int_{\mathbb{R}^{N}}[f(\omega_{0}^{R}+\omega_{y}^{R})-f(\omega_{0}%
^{R})]\,\omega_{0}^{R}\\
&  \quad-\int_{\mathbb{R}^{N}}[f(\omega_{0}^{R}+\omega_{y}^{R})-f(\omega
_{y}^{R})]\,\omega_{y}^{R}\text{.}%
\end{align*}

Set $h(s):=A_{1}\min\{\left\vert s\right\vert ^{p-2},\left\vert s\right\vert
^{q-2}\}\leq\left\vert s\right\vert ^{2^{\ast}-2}.$ Using inequality
(\ref{eq:f}) we get that%
\begin{align*}
\int_{\mathbb{R}^{N}}\left\vert f(\omega_{0}^{R}+\omega_{y}^{R})-f(\omega
_{0}^{R})\right\vert \,\omega_{0}^{R}  &  \leq\int_{\mathbb{R}^{N}}%
h(\omega_{0}^{R}+\omega_{y}^{R})\omega_{y}^{R}\omega_{0}^{R}\\
&  \leq\left\vert \omega_{0}^{R}+\omega_{y}^{R}\right\vert _{2^{\ast}%
}^{2^{\ast}-2}\left(  \int_{\mathbb{R}^{N}}\left(  \omega_{y}^{R}\right)
^{2^{\ast}/2}\left(  \omega_{0}^{R}\right)  ^{2^{\ast}/2}\right)  ^{2/2^{\ast
}}\\
&  \leq C\left(  \int_{\mathbb{R}^{N}}\left(  \omega_{y-y_{0}}^{R}\right)
^{2^{\ast}/2}\omega^{2^{\ast}/2}\right)  ^{2/2^{\ast}}%
\end{align*}
and, similarly,%
\[
\int_{\mathbb{R}^{N}}\left\vert f(\omega_{0}^{R}+\omega_{y}^{R})-f(\omega
_{y}^{R})\right\vert \,\omega_{y}^{R}\leq C\left(  \int_{\mathbb{R}^{N}}%
\omega^{2^{\ast}/2}\left(  \omega_{y-y_{0}}^{R}\right)  ^{2^{\ast}/2}\right)
^{2/2^{\ast}}.
\]
From the above inequalities we conclude that $J_{0}(\omega_{0}^{R}+\omega
_{y}^{R})=o_{R}(1)$, where $o_{R}(1)\rightarrow0$ as $R\rightarrow\infty,$
uniformly in $y\in\partial B_{2}(y_{0}).$ Hence, for $\lambda=\frac{1}{2},$ we
get that%
\[
J_{V}(2\,z_{\lambda,\,y}^{R})=J_{V}(\omega_{0}^{R}+\omega_{y}^{R}%
)=J_{0}(\omega_{0}^{R}+\omega_{y}^{R})+\int_{\mathbb{R}^{N}}V\left(
\omega_{0}^{R}+\omega_{y}^{R}\right)  ^{2}=o_{R}(1).
\]
This yields the claim.
\end{proof}

The proof of the next result follows that of Lemma 2.1 in \cite{ew}.

\begin{lemma}
\label{new}For each $a>0$ there exists $C_{a}\geq0$ such that
\[
F(s+t)-F(s)-F(t)-f(s)t-f(t)s\geq-C_{a}(st)^{1+\frac{\nu}{2}}%
\]
for all $s,t\in\lbrack0,a]$ and $\nu\in(0,q-2)$.
\end{lemma}

\begin{proof}
The inequality is clearly satisfied if $s=0$ or $t=0$. Assumption $(f_{3})$
implies that $f$ is increasing for $s>0$. Therefore,
\[
F(s+t)-F(s)=\int_{s}^{s+t}f(\zeta)\,\mathrm{d}\zeta\geq f(s)t
\]
for all $s,t>0.$ Moreover, by $(f_{2}),$ we have that $f(s)=o(|s|^{1+\nu}%
)\ $for any $\nu\in(0,q-2).$ Hence, there exists $M_{a}>0$ such that
$f(s)\leq{M}_{a}s^{1+\nu}$ for all $s\in\left[  0,a\right]  $. It follows
that
\begin{align*}
&  F(s+t)-F(s)-F(t)-f(s)t-f(t)s\geq-F(t)-f(t)s\\
&  \geq-{M}_{a}\int_{0}^{t}\zeta^{1+\nu}\mathrm{d}\zeta-M_{a}st^{1+\nu}%
=-M_{a}\left(  \frac{t^{2+\nu}}{2+\nu}+st^{1+\nu}\right)
\end{align*}
for all $s,t\in\lbrack0,a].$ So, if $t\leq s$, we get that%
\[
F(s+t)-F(s)-F(t)-f(s)t-f(t)s\geq-\frac{3}{2}{M}_{a}(st)^{1+\frac{\nu}{2}}.
\]
As this expression is symmetric in $s$ and $t$, it holds true also when $s\leq
t$, and the proof is complete.
\end{proof}

With the previous lemmas on hand, we now prove the following estimate.

\begin{proposition}
\label{prop:main}There exists $R_{1}\geq1$ and, for each $R>R_{1},$ a number
$\eta_{R}>0$ such that%
\[
I_{V}(T_{\lambda,y}^{R}z_{\lambda,y}^{R})\leq2c_{0}-\eta_{R}%
\]
for all $\lambda\in\lbrack0,1]$ and all $y\in\partial B_{2}(y_{0})$.
\end{proposition}

\begin{proof}
To simplify the notation, let us set
\[
s:=T_{\lambda,y}^{R}\lambda,\quad t:=T_{\lambda,y}^{R}(1-\lambda),\quad
\omega_{0}:=\omega_{0}^{R},\quad\omega_{y}:=\omega_{y}^{R}.
\]
Then $s,t\in\lbrack0,T_{0})$ if $R\geq R_{0}$, with $R_{0}\geq1$ and $T_{0}>2$
as in Lemma \ref{projection}, and
\begin{align}
&  I_{V}(T_{\lambda,y}^{R}z_{\lambda,y}^{R})=I_{V}(s\omega_{0}+t\omega
_{y})\nonumber\\
&  =\frac{s^{2}}{2}\int_{\mathbb{R}^{N}}|\nabla\omega_{0}|^{2}+\frac{s^{2}}%
{2}\int_{\mathbb{R}^{N}}V\omega_{0}^{2}+\frac{t^{2}}{2}\int_{\mathbb{R}^{N}%
}|\nabla\omega_{y}|^{2}+\frac{t^{2}}{2}\int_{\mathbb{R}^{N}}V\omega_{y}%
^{2}\nonumber\\
&  \quad+st\int_{\mathbb{R}^{N}}\nabla\omega_{0}\cdot\nabla\omega_{y}%
+st\int_{\mathbb{R}^{N}}V\omega_{0}\omega_{y}-\int_{\mathbb{R}^{N}}%
F(s\omega_{0}+t\omega_{y})\nonumber\\
&  =\frac{s^{2}}{2}\int_{\mathbb{R}^{N}}|\nabla\omega_{0}|^{2}-\int
_{\mathbb{R}^{N}}F(s\omega_{0})+\frac{t^{2}}{2}\int_{\mathbb{R}^{N}}%
|\nabla\omega_{y}|^{2}-\int_{\mathbb{R}^{N}}F(t\omega_{y})\label{eq1}\\
&  \quad+\frac{s^{2}}{2}\int_{\mathbb{R}^{N}}V\omega_{0}^{2}+\frac{t^{2}}%
{2}\int_{\mathbb{R}^{N}}V\omega_{y}^{2}+st\int_{\mathbb{R}^{N}}V\omega
_{0}\omega_{y}\label{eq2}\\
&  \quad+st\int_{\mathbb{R}^{N}}\nabla\omega_{0}\cdot\nabla\omega
_{y}\label{eq3}\\
&  \quad-\int_{\mathbb{R}^{N}}\left[  F(s\omega_{0}+t\omega_{y})-F(s\omega
_{0})-F(t\omega_{y})-f(s\omega_{0})t\omega_{y}-f(t\omega_{y})s\omega
_{0}\right] \label{eq4}\\
&  \quad-\int_{\mathbb{R}^{N}}f(s\omega_{0})t\omega_{y}-\int_{\mathbb{R}^{N}%
}f(t\omega_{y})s\omega_{0} \label{eq5}%
\end{align}
Next, we estimate each of the numbered lines. As $\omega_{0}$ and $\omega_{y}$
are ground states of the limit problem (\ref{pinf}), Lemma \ref{lem:nehari}%
(d)\ yields%
\[
(\ref{eq1})=I_{0}(s\omega_{0})+I_{0}(t\omega_{y})\leq I_{0}(\omega_{0}%
)+I_{0}(\omega_{y})=2c_{0}.
\]
From Lemma \ref{lem:V} and estimates (\ref{eq:epsilon}) we get that
\[
(\ref{eq2})\leq CR^{-\tau}=o(\varepsilon_{R}).
\]
Lemma \ref{new}\ with $\nu\in(\frac{2}{N-2},q-2)$ and Lemma \ref{power}(a)
with $\alpha=\beta=(1+\frac{\nu}{2})(N-2)$, imply that, for some $\mu>N-2$,
\begin{align*}
(\ref{eq4})  &  =-\int_{\mathbb{R}^{N}}\left[  F(s\omega_{0}+t\omega
_{y})-F(s\omega_{0})-F(t\omega_{y})-f(s\omega_{0})t\omega_{y}-f(t\omega
_{y})s\omega_{0}\right] \\
&  \leq C\left\vert st\right\vert ^{1+\frac{\nu}{2}}\int_{\mathbb{R}^{N}%
}(\omega_{0}\omega_{y})^{1+\frac{\nu}{2}}\leq CR^{-\mu}=o(\varepsilon_{R}).
\end{align*}
We write the sum of the remaining terms as
\begin{align}
(\ref{eq3})+(\ref{eq5})  &  =\frac{t}{2}\int_{\mathbb{R}^{N}}[sf(\omega
_{0})-f(s\omega_{0})]\omega_{y}+\frac{s}{2}\int_{\mathbb{R}^{N}}[tf(\omega
_{y})-f(t\omega_{y})]\omega_{0}\nonumber\\
&  -\frac{1}{2}\int_{\mathbb{R}^{N}}f(s\omega_{0})\,t\omega_{y}-\frac{1}%
{2}\int_{\mathbb{R}^{N}}f(t\omega_{y})\,s\omega_{0}.\nonumber
\end{align}
By Lemma \ref{tvm} there exists a constant $C>0$ such that
\[
\frac{t}{2}\int_{\mathbb{R}^{N}}\left\vert sf(\omega_{0})-f(s\omega
_{0})\right\vert \omega_{y}+\frac{s}{2}\int_{\mathbb{R}^{N}}\left\vert
tf(\omega_{y})-f(t\omega_{y})\right\vert \omega_{0}\leq C\left(  \left\vert
s-1\right\vert +\left\vert t-1\right\vert \right)  \varepsilon_{R},
\]
for all $s,t\in\lbrack0,T_{0}],$ $y\in\partial B_{2}(y_{0})$ and $R\geq R_{0}%
$. Moreover, Lemma \ref{lowerbound}, yields a constant $C_{0}>0$ such that
\[
\frac{1}{2}\int_{\mathbb{R}^{N}}f(s\omega_{0})\,t\omega_{y}+\frac{1}{2}%
\int_{\mathbb{R}^{N}}f(t\omega_{y})\,s\omega_{0}\geq C_{0}\varepsilon_{R}%
\]
for all $s,t\geq\frac{1}{2}$, $y\in\partial B_{2}(y_{0})$ and $R\geq R_{0}$.
By Lemma \ref{lem:lambda1/2}, if $\lambda=\frac{1}{2},$ then $s,t\rightarrow1$
as $R\rightarrow\infty.$ Therefore, there exist $R_{1}\geq R_{0}$ and
$\delta\in(0,\frac{1}{2})$ such that
\[
(\ref{eq3})+(\ref{eq5})\leq-\frac{C_{0}}{2}\varepsilon_{R}%
\]
for all $\lambda\in\lbrack\frac{1}{2}-\delta,\frac{1}{2}+\delta],$
$y\in\partial B_{2}(y_{0})$ and $R\geq R_{1}.$ Summing up, we have shown that
\begin{equation}
I_{V}(s\omega_{0}+t\omega_{y})\leq2c_{0}-\frac{C_{0}}{2}\varepsilon
_{R}+o(\varepsilon_{R}) \label{ineq1}%
\end{equation}
for all $\lambda\in\lbrack\frac{1}{2}-\delta,\frac{1}{2}+\delta]$,
$y\in\partial B_{2}(y_{0}),$ $R\geq R_{1}.$

On the other hand, by Lemma \ref{lem:nehari}(d), there exists $\gamma
\in(0,c_{0})$ such that
\[
(\ref{eq1})=I_{0}(s\omega_{0})+I_{0}(t\omega_{y})\leq2c_{0}-\gamma
\]
for all $\lambda\in\lbrack0,\frac{1}{2}-\delta]\cup\lbrack\frac{1}{2}%
+\delta,1]$, $y\in\partial B_{2}(y_{0})$ and $R$ sufficiently large. Since
$(\ref{eq2})+\cdots+(\ref{eq5})\leq O(\varepsilon_{R}),$ we conclude that%
\begin{equation}
I_{V}(s\omega_{0}+t\omega_{y})\leq2c_{0}-2\gamma+O(\varepsilon_{R})
\label{ineq2}%
\end{equation}
for all $\lambda\in\lbrack0,\frac{1}{2}-\delta]\cup\lbrack\frac{1}{2}%
+\delta,1]$, $y\in\partial B_{2}(y_{0})$ and $R$ sufficiently large.

Inequalities (\ref{ineq1}) and (\ref{ineq2}), together, yield the statement of
the proposition.
\end{proof}

\begin{lemma}
\label{energy_on_boundary} For any $\delta>0$, there exists $R_{2}>0$ such
that
\[
I_{V}(T_{\lambda,y}^{R}z_{\lambda,y}^{R})<c_{0}+\delta
\]
for $\lambda=1$ and every $y\in\partial B_{2}(y_{0})$ and $R\geq R_{2}.$ In
particular, $c_{V}\leq c_{0}.$
\end{lemma}

\begin{proof}
By Lemma \ref{projection}, $T_{\lambda,y}^{R}$ is bounded uniformly in
$\lambda,y$ and $R$. So, from Lemmas \ref{lem:nehari}(d) and \ref{lem:V}, we
obtain that%
\[
I_{V}(T_{1,y}^{R}z_{1,y}^{R})=I_{0}(T_{1,y}^{R}\omega_{y}^{R})+(T_{1,y}%
^{R})^{2}\int_{\mathbb{R}^{N}}V\left(  \omega_{y}^{R}\right)  ^{2}\leq
c_{0}+o_{R}(1),
\]
where $o_{R}(1)\rightarrow0$ as $R\rightarrow\infty,$ uniformly in
$y\in\partial B_{2}(y_{0}),$ and the claim is proved.
\end{proof}

For $c\in\mathbb{R},$ set%
\[
I_{V}^{c}:=\{u\in D^{1,2}(\mathbb{R}^{N}):I_{V}(u)\leq c\}.
\]

\begin{lemma}
[Deformation]\label{lem:deformation}If $c_{V}$ is not attained by $I_{V}$ on
$\mathcal{N}_{V},$ then $c_{V}=c_{0}.$ If, moreover, $I_{V}$ does not have a
critical value in $(c_{0},2c_{0})$ then, for any given $\delta,\eta\in
(0,\frac{c_{0}}{4}),$ there exists a continuous function%
\[
\pi:\mathcal{N}_{V}\cap I_{V}^{2c_{0}-\eta}\rightarrow\mathcal{N}_{V}\cap
I_{V}^{c_{0}+\delta}%
\]
such that $\pi(u)=u$ for all $u\in\mathcal{N}_{V}\cap I_{V}^{c_{0}+\delta}.$
\end{lemma}

\begin{proof}
If $c_{V}$ is not attained, Corollary \ref{cor:compactness}(a)\ and Lemma
\ref{energy_on_boundary}\ imply that $c_{V}=c_{0}.$

Recall that $\mathcal{N}_{V}$ is a $\mathcal{C}^{1}$-manifold. From Lemma
\ref{lem:aproxPS} and Corollary \ref{cor:compactness}(b) we have that the
following statement is true: If $\sigma_{k}\in\mathbb{R}$ and $u_{k}%
\in\mathcal{M}_{\sigma_{k}}$ are such that $\sigma_{k}\rightarrow0,$
$I_{V}(u_{k})\rightarrow d\in(c_{0},2c_{0})$ and, either $\nabla
_{\mathcal{M}_{\sigma_{k}}}I_{V}(u_{k})\rightarrow0$ or $\nabla J_{V}%
(u_{k})\rightarrow0$, then $(u_{k})$ contains a convergent subsequence. (In
fact, Lemma \ref{lem:aproxPS} says that $(\nabla J_{V}(u_{k}))$ must be
bounded away from $0$). This allows us to apply Theorem 2.5 in \cite{b} to
conclude that there exists $\hat{\varepsilon}>0$ such that, for each
$\varepsilon\in(0,\hat{\varepsilon}),$ there exists a homeomorphism
$\phi:\mathcal{N}_{V}\rightarrow\mathcal{N}_{V}$ such that

\begin{enumerate}
\item[(i)] $\phi(u)=u$ if $I_{V}(u)\notin\lbrack d-\varepsilon,d+\varepsilon
],$

\item[(ii)] $I_{V}(\phi(u))\leq I_{V}(u)$ for all $u\in\mathcal{N}_{V},$

\item[(iii)] $I_{V}(\phi(u))\leq d-\varepsilon$ for all $u\in\mathcal{N}_{V}$
with $I_{V}(u)\leq d+\varepsilon.$
\end{enumerate}

Our claim follows easily from this fact.
\end{proof}

Let $\mathfrak{b}:L^{2^{\ast}}(\mathbb{R}^{N})\smallsetminus\{0\}\rightarrow
\mathbb{R}^{N}$ be a barycenter map, i.e., a continuous map such that%
\begin{equation}
\mathfrak{b}(u(\cdot-y))=\mathfrak{b}(u)+y\text{\qquad and\qquad}%
\mathfrak{b}(u\circ\Theta^{-1})=\Theta(\mathfrak{b}(u)) \label{bar}%
\end{equation}
for all $u\in L^{2^{\ast}}(\mathbb{R}^{N})\smallsetminus\{0\}$ and
$y\in\mathbb{R}^{N},$ and every linear isometry $\Theta$ of $\mathbb{R}^{N}$;
see \cite{bw,cp}. Note that $\mathfrak{b}(u)=0$ if $u$ is radial.

\begin{lemma}
\label{lem:bar}If $c_{V}$ is not attained by $I_{V}$ on $\mathcal{N}_{V},$
then there exists $\delta>0$ such that
\[
\mathfrak{b}(u)\neq0\qquad\forall u\in\mathcal{N}_{V}\cap I_{V}^{c_{0}+\delta
}.
\]

\end{lemma}

\begin{proof}
The proof is the same as that of Lemma 3.11 in \cite{cm}.
\end{proof}

\smallskip

\begin{proof}
[Proof of Theorem \ref{thm:main}]If $c_{V}$ is attained by $I_{V}$ at some
$u\in\mathcal{N}_{V},$ then $u$ is a nontrivial solution of problem
(\ref{prob}). So assume that $c_{V}$ is not attained. Then, by Lemma
\ref{lem:deformation}, $c_{V}=c_{0}$. We will show that $I_{V}$ has a critical
value in $(c_{0},2c_{0}).$

Lemma \ref{lem:bar} allows us to choose $\delta\in(0,\frac{c_{0}}{4})$ such
that
\[
\mathfrak{b}(u)\neq0\qquad\forall u\in\mathcal{N}_{V}\cap I_{V}^{c_{0}+\delta}%
\]
and, by Proposition \ref{prop:main}\ and Lemma \ref{energy_on_boundary},\ we
may choose $R\geq1$ and $\eta\in(0,\frac{c_{0}}{4})$\ such that%
\[
I_{V}(T_{\lambda,y}^{R}z_{\lambda,y}^{R})\leq\left\{
\begin{array}
[c]{ll}%
2c_{0}-\eta & \text{for all }\lambda\in\lbrack0,1]\text{ and all }y\in\partial
B_{2}(y_{0}),\\
c_{0}+\delta & \text{for }\lambda=1\text{ and all }y\in\partial B_{2}(y_{0}).
\end{array}
\right.
\]
Define $\iota:B_{2}(y_{0})\rightarrow\mathcal{N}_{V}\cap I_{V}^{2c_{0}-\eta}$
by%
\[
\iota((1-\lambda)y_{0}+\lambda y):=T_{\lambda,y}^{R}z_{\lambda,y}%
^{R},\text{\qquad with \ }\lambda\in\lbrack0,1],\text{ }y\in\partial
B_{2}(y_{0}).
\]
Arguing by contradiction, assume that $I_{V}$ does not have a critical value
in $(c_{0},2c_{0}).$ Then, by Lemma \ref{lem:deformation}, there exists a
continuous function%
\[
\pi:\mathcal{N}_{V}\cap I_{V}^{2c_{0}-\eta}\rightarrow\mathcal{N}_{V}\cap
I_{V}^{c_{0}+\delta}%
\]
such that $\pi(u)=u$ for all $u\in\mathcal{N}_{V}\cap I_{V}^{c_{0}+\delta}.$
The function $\psi:B_{2}(y_{0})\rightarrow\partial B_{2}(y_{0})$ given by%
\[
\psi(x):=2\frac{\left(  \mathfrak{b}\circ\pi\circ\iota\right)  (x)}{\left\vert
\left(  \mathfrak{b}\circ\pi\circ\iota\right)  (x)\right\vert }%
\]
is well defined and continuous, and $\psi(y)=y$ for every $y\in\partial
B_{2}(y_{0}).$ This is a contradiction. Therefore, $I_{V}$ must have a
critical point $u\in\mathcal{N}_{V}$ with $I_{V}(u)\in(c_{0},2c_{0}).$

By Lemma \ref{lem:sign}, $u$ does not change sign and, since $f$ is odd, $-u$
is also a solution of (\ref{prob}). This proves that problem (\ref{prob}) has
a positive solution.
\end{proof}

\end{document}